\documentclass[final,onefignum,onetabnum]{siamonline220329}

\usepackage{lipsum}
\usepackage{amsfonts}
\usepackage{graphicx}
\usepackage{subfigure} 
\usepackage{epstopdf}
\usepackage{algorithmic}
\usepackage{booktabs}
\usepackage{amsmath}
\usepackage{url}

\ifpdf
  \DeclareGraphicsExtensions{.eps,.pdf,.png,.jpg}
\else
  \DeclareGraphicsExtensions{.eps}
\fi

\usepackage{enumitem}
\setlist[enumerate]{leftmargin=.5in}
\setlist[itemize]{leftmargin=.5in}

\newsiamremark{remark}{Remark}
\newsiamremark{hypothesis}{Hypothesis}
\crefname{hypothesis}{Hypothesis}{Hypotheses}
\newsiamthm{claim}{Claim}

\headers{Sliding at first order}{L. Bao, J. Lu, S. Ying, and S. Sommer}

\title{Sliding at first order: Higher-order momentum distributions for discontinuous image registration\thanks{Submitted to the editors DATE.
\funding{The work of the first and third authors was supported by the National Natural Science Foundation of China (No. 11971296) and the National Key R\&D Program of China (No. 2021YFA1003004). The work of the fourth author was supported by the Villum Foundation grant 40582, the Novo Nordisk Foundation grant NNF18OC0052000, and the UCPH Data+ strategy funds.}}}

\author{
Lili Bao\footnotemark[2]~\footnotemark[3]\thanks{Department of Mathematics, Shanghai University, Shanghai 200444, P. R. China.}
\and Jiahao Lu\footnotemark[3]\thanks{Department of Computer Science, University of Copenhagen, Universitetsparken 1, Copenhagen 2100, Denmark.}
\and Shihui Ying\footnotemark[2]~\thanks{Corresponding author. \email{shying@shu.edu.cn},~ \email{sommer@di.ku.dk}.}
\and Stefan Sommer\footnotemark[3]~\footnotemark[4]
}

\usepackage{amsopn}

\ifpdf
\hypersetup{
  pdftitle={Sliding at first order: Higher-order momentum distributions for discontinuous image registration},
  pdfauthor={L. Bao, J. Lu, S. Ying, and S. Sommer}
}
\fi

\begin{document}

\maketitle

\begin{abstract}
In this paper, we propose a new approach to deformable image registration that captures sliding motions. The large deformation diffeomorphic metric mapping (LDDMM) registration method faces challenges in representing sliding motion since it per construction generates smooth warps. To address this issue, we extend LDDMM by incorporating both zeroth- and first-order momenta with a non-differentiable kernel. This allows to represent both discontinuous deformation at switching boundaries and diffeomorphic deformation in homogeneous regions. We provide a mathematical analysis of the proposed deformation model from the viewpoint of discontinuous systems. To evaluate our approach, we conduct experiments on both artificial images and the publicly available DIR-Lab 4DCT dataset. Results show the effectiveness of our approach in capturing plausible sliding motion.
\end{abstract}



\section{Introduction}
\label{sec: Introduction}
Image registration is a widely used technique in computer vision and medical image processing with the goal of finding reasonable spatial deformations between two or more images \cite{sotiras_deformable_2013,oliveira_medical_2014}. However, it can be a challenging task due to the highly ill-posed nature of the problem, particularly when addressing sliding motion which results in discontinuous deformation.  

Current sliding motion registration methods often employ displacement-based techniques 
and special constraints are designed for the displacement field. Such approaches can be limited when applied to e.g. lung registration during inspiration and expiration, where large sliding motion occurs and invertibility of the deformations is crucial. The Large Deformation Diffeomorphic Metric Mapping (LDDMM) framework in focus here can handle large deformations and ensures invertibility. The velocity fields are smooth and point-supported momentum models local translations (illustrated in \cref{fig: Gaussian kernel zeroth simulation}). 
The first-order momentum introduced in \cite{sommer_higher-order_2013} allows for modelling local linear deformations (see \cref{fig: Gaussian kernel first simulation}) with smooth kernels. In contrast, non-differentiable kernels, as shown in \cite{jud_sparse_2016}, can deal with discontinuous deformation.
As seen in \cref{fig: Multiplicative Wendland kernel first simulation}, encoding the deformation with the non-differentiable Wendland kernel with first-order allows for simulating sliding motion.

Based on these observations, we propose a sliding motion registration method by incorporating a compactly supported non-differentiable kernel and its corresponding first-order momentum based on the LDDMM framework.
We show how this approach allows discontinuous deformation at switching boundaries while maintaining diffeomorphic deformation in homogeneous regions. We also analyze and illustrate how the trajectory changes when it hits the switching boundaries, resulting in discontinuous deformation.

\begin{figure}[htbp]
\centering
\subfigure[]{
\includegraphics[width=0.3\textwidth]{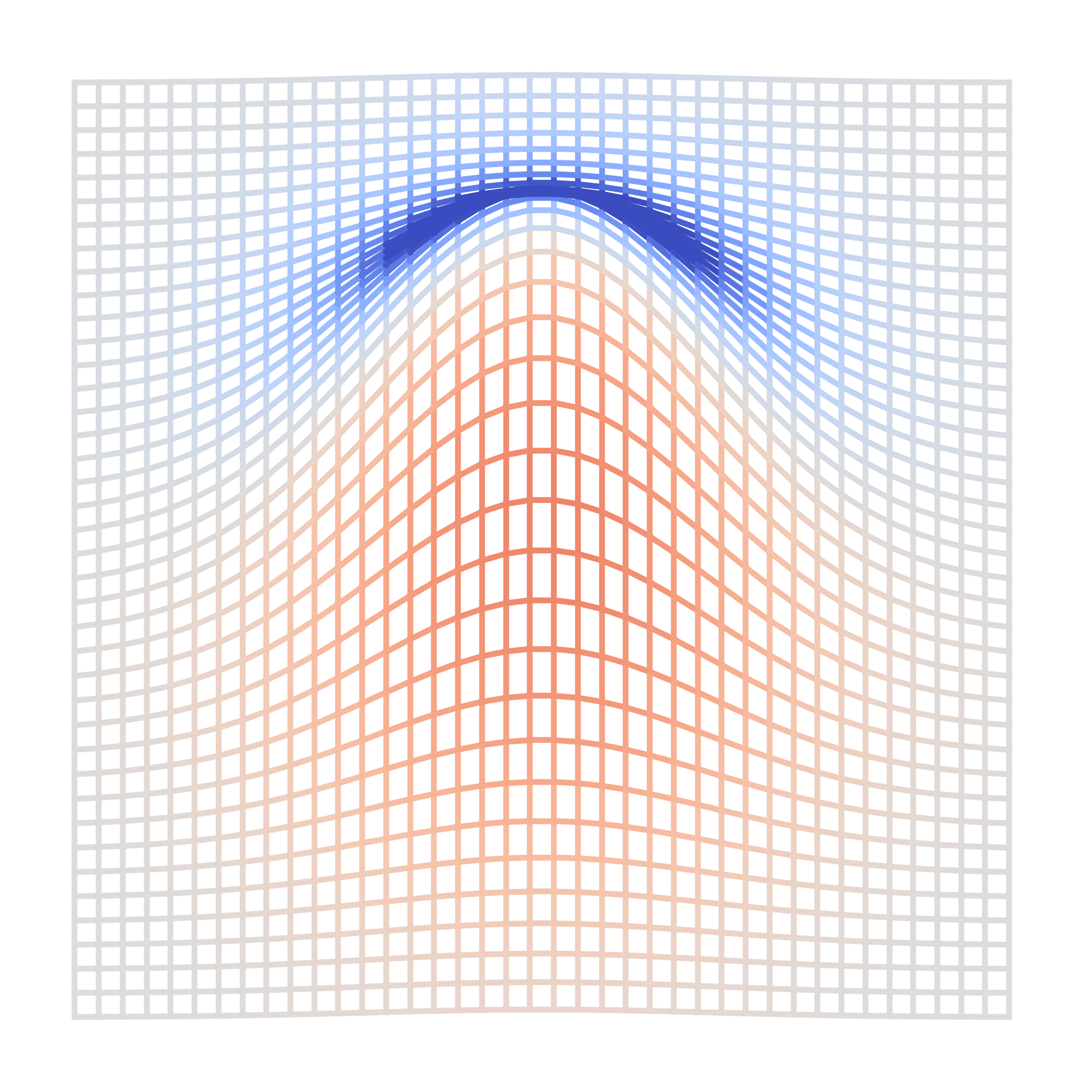}
\label{fig: Gaussian kernel zeroth simulation}}
\centering
\subfigure[]{
\includegraphics[width=0.3\textwidth]{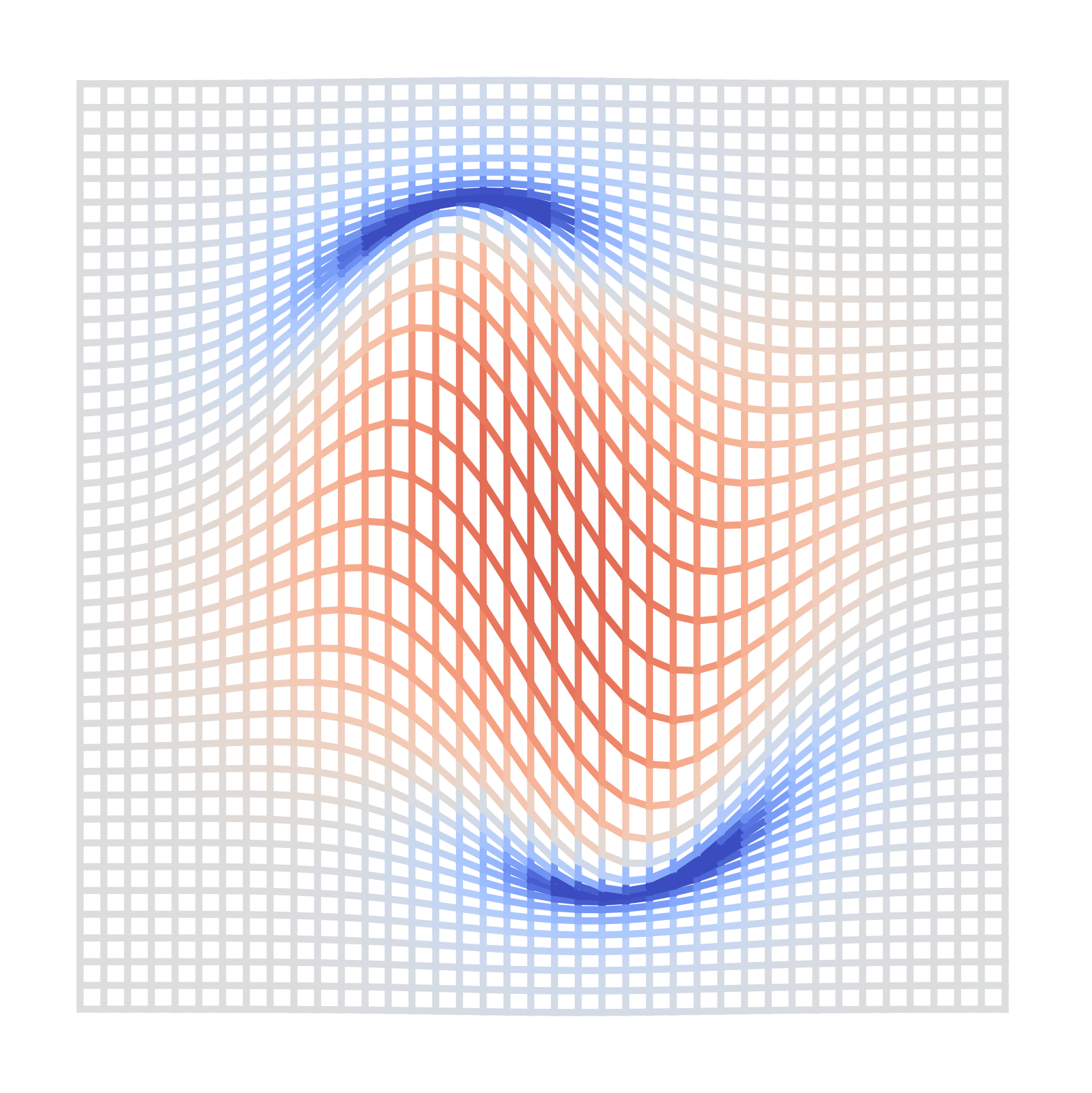}
\label{fig: Gaussian kernel first simulation}}
\centering
\subfigure[]{
\includegraphics[width=0.3\textwidth]{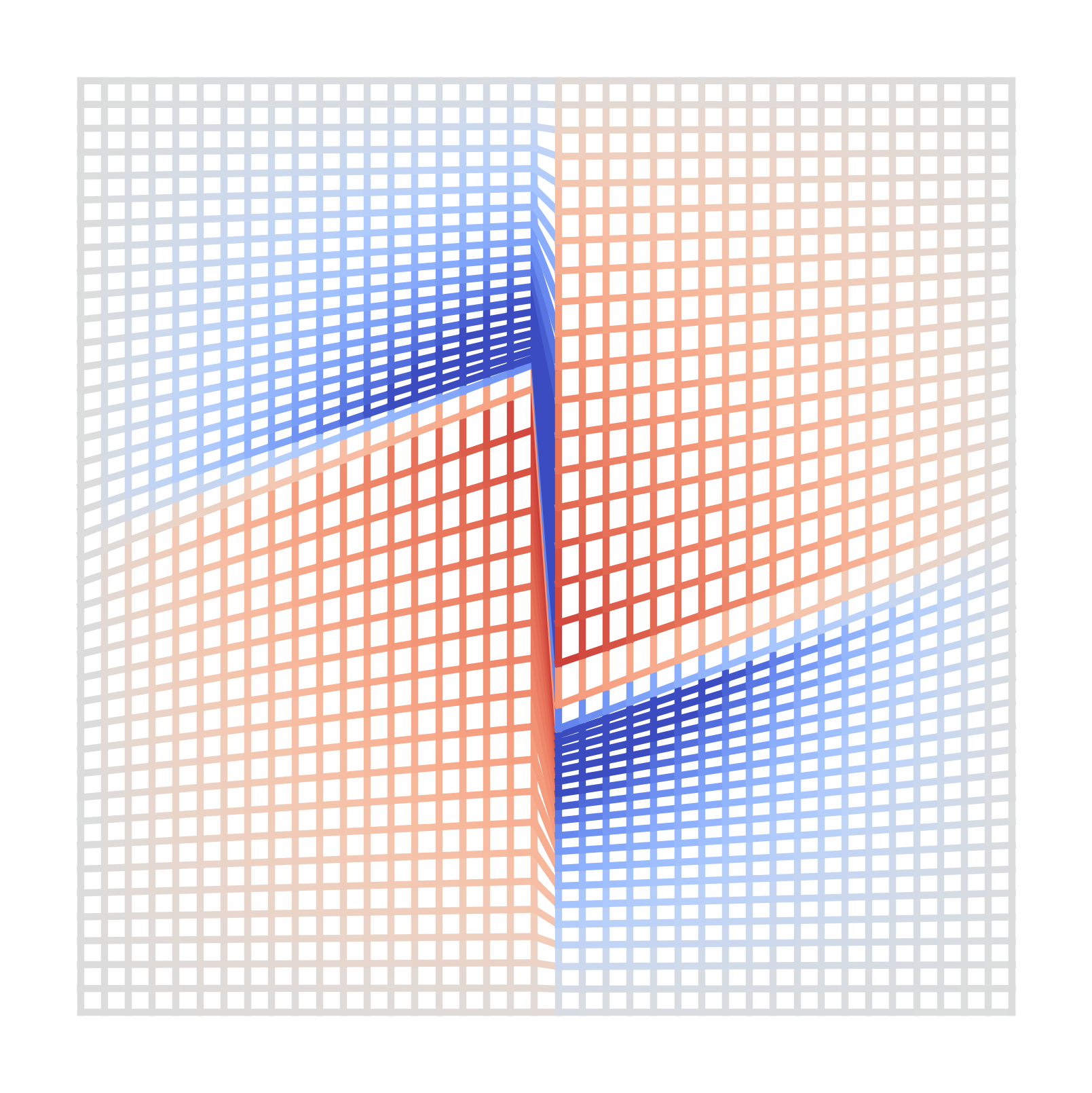}
\label{fig: Multiplicative Wendland kernel first simulation}}
\caption{The deformation encoded by smooth and non-smooth kernels with different orders. \subref{fig: Gaussian kernel zeroth simulation} Gaussian kernel with zeroth order encodes local translation; 
\subref{fig: Gaussian kernel first simulation} Gaussian kernel with first order encodes local smooth sheering; \subref{fig: Multiplicative Wendland kernel first simulation} Non-differentiable Wendland kernel with first order encodes local non-smooth sliding motion.}
\label{fig: simulate}
\end{figure}

\subsection{Background and related work}
\label{subsec: Related work}
The incorporation of sliding motion in image registration is particularly important in medical image analysis, particularly in the context of lung registration. Aligning images of the same patient acquired at different times or with different devices is crucial for disease diagnosis and monitoring progression. Additionally, accurately estimating lung breathing motion can improve the precision of radiotherapy treatment \cite{schmidt-richberg_registration_2014}. 
It would benefit from incorporating sliding motion when doing lung registration, because during breathing, the lungs move against the chest wall across the pleura, resulting in natural sliding motion (as seen in \cref{fig: sliding illustration}). Thus sliding motion registration is a crucial aspect of medical image processing. 

 \begin{figure}[htbp]
\centering
\includegraphics[width=0.7\textwidth]{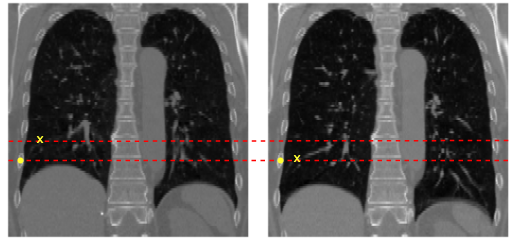}
\caption{\label{fig: sliding illustration}Illustration of the sliding motion for CT lung images from the end exhalation (left) to the end inhalation (right) with two landmarks. The motion of the two landmarks clearly illustrates the sliding motion at the lung boundary. }
\end{figure}

Over the years, numerous methods have been proposed to address the challenge of sliding motion registration. These methods can broadly be divided into two categories: parametric models and non-parametric methods.

Parametric image registration techniques involve finitely parameterizing the deformation field using a small number of parameters, such as control points or coefficients. One commonly used parametric technique is the Free-Form Deformation (FFD) method, which models the deformation field as a combination of basis functions and a set of control points, such as B-Splines \cite{rueckert_nonrigid_1999} and thin-plate splines \cite{bookstein_principal_1989}. When addressing sliding motion, special constraints should be added as the deformation is no longer smooth. For example, in \cite{yoshida_total_2014, vishnevskiy_isotropic_2017} the authors imposed anisotropic and isotropic total variation regularizations respectively to handle non-smooth motion. In \cite{gong_locally_2020}, the author presents a locally adaptive total p-variation to estimate the sliding motion. 
These parametric techniques \cite{rueckert_nonrigid_1999, hua_multiresolution_2017, vishnevskiy_isotropic_2017} are computationally efficient due to their limited degrees of freedom. Their limitations include a trade-off between increasing the number of control points for improved accuracy and increased computational cost, and difficulty in representing complex deformation.

An alternative approach to parametric techniques is non-parametric techniques, which represent the deformation field as a continuous function instead of a finite number of parameters. A popular approach is the variational-based methods, where the deformation field is obtained by minimizing an energy functional that encodes the similarity between images, and a regularization term that enforces certain properties of the deformation field to exclude suboptimal solutions \cite{modersitzki_numerical_2003}. Various regularizations, such as diffusion, elastic \cite{modersitzki_numerical_2003}, and curvature \cite{fischer_curvature_2003} have been proposed to obtain smooth deformation. When dealing with non-smooth sliding motion, special regularizers have been proposed. For example, total variation has been used to describe the discontinuity of the deformation field in \cite{frohn-schauf_multigrid_2008}, while locally adaptive bilateral filters have been used to regularize the estimated deformation fields in \cite{papiez_implicit_2014}.
Direction-dependent regularization methods have been employed in \cite{fu_adaptive_2018,schmidt-richberg_estimation_2012},
where the deformation field is decomposed into normal and tangential directions and regularized separately. 
While these methods that encode the deformation in displacement fields can provide plausible deformation fields for small deformations, they struggle to handle large deformations and cannot ensure the invertibility of the deformations, leading to artifacts such as tearing and folding.

To overcome these limitations and ensure the invertibility of the deformations, an alternative approach to non-parametric image registration is to model the deformation field as a flow generated by a smooth velocity field, which can effectively handle large deformations while maintaining diffeomorphism \cite{younes_shapes_2019}. This approach, known as the Large Deformation Diffeomorphic Metric Mapping (LDDMM) framework \cite{beg_computing_2005}, has strong theoretical foundations in Lie group theory and evolution equations in physical modeling, providing a solid foundation for accurate registration results \cite{szekely_multi-scale_2011}.
In order to ensure the invertibility of the deformation when dealing with sliding motion, piecewise-diffeomorphic models have been proposed in \cite{risser_piecewise-diffeomorphic_2013,risser_dieomorphic_2011}, which allow non-smooth deformations at the boundary while preserving invertibility throughout the whole domain, providing a more practical solution while still lacking some theoretical analysis.

Here, we propose an extension of the LDDMM registration framework by incorporating non-differentiable kernel and its first-order momentum.
This method allows for effectively handling sliding motion while preserving the invertibility of the deformation, and we show that much of the structure leading to optimal paths in LDDMM are preserved.

\subsection{Content and Outline}
\label{subsec: Content and Outline}
We start by describing the standard LDDMM registration framework and the mathematical foundation of the method. We then give a detailed explanation and analysis of the model with higher order momenta and non-differentiable kernel. Next follows the evaluation of synthetic images and the publicly available DIR-Lab 4DCT dataset.
The main contributions of this paper are:
\begin{itemize}
    \item Extending zeroth-order momentum in the standard LDDMM to zeroth- and first-order momenta with non-differentiable kernels to capture a more comprehensive representation of the local deformation.
    \item Using a non-differentiable multiplicative Wendland kernel for discontinuous deformation.
    \item Using a sparsity prior to force compact representation of the deformation across orders.
    \item Providing a mathematical analysis of the proposed discontinuous deformation from the viewpoint of discontinuous systems.
\end{itemize}

\section{Standard LDDMM Registration}
\label{sec: Standard LDDMM Registration}
In this section, we outline the LDDMM registration framework.
Given a template image $I_0$ and a reference image $I_1$ defined on the image domain $\Omega\subset R^{d}~(d\in\{2,3\})$,
image registration seeks a reasonable deformation field $\varphi:\Omega\rightarrow\Omega$, such that the reference image $I_1$ and the warped template image $\varphi \cdot I_0 = I_0\circ\varphi^{-1}$ are as similar as possible. Therefore, the image registration between the template image $I_0$ and the reference image $I_1$ can be formulated as the following minimizing problem:
\begin{equation}
   \varphi^\ast =\mathop{\arg\min}_{\varphi} E(\varphi):=  E_S(\varphi \cdot I_0, I_1)+E_R(\varphi),
\label{eq: functional} 
\end{equation}
where the first term $E_S$ is a similarity term that measures the similarity between the two images, and $E_R$ is a regularization term that penalizes undesirable or implausible solutions.

The selection of the similarity measure $E_S$ varies depending on the images being registered. Some commonly used measures include the sum of squared differences (SSD)\cite{hill_medical_2001}, mutual information (MI)\cite{maes_multimodality_1997}, normalized cross-correlation (NCC)\cite{wells_correlation_1998}, and normalized gradient fields (NGF)\cite{modersitzki_fair_2009}. Each measure has its own advantages and disadvantages, thus the choice of similarity measure should be based on the specific application and the desired level of accuracy.  

Now we turn to the regularization term $E_R$ based on fluid approaches, where the deformation is modeled as a flow generated by a smooth, time-dependent vector field $v =(v(t), t\in [0,1])$ through a differential equation:
\begin{equation}
    \frac{\partial \varphi(t, x)}{\partial t} = v(t, \varphi(t, x)),\ \varphi(0, x) = x, \forall x \in \Omega.
    \label{eq: differential equation}
\end{equation}
This equation generates a path of diffeomorphism $t \mapsto \varphi_{0t}^v$ starting at the beginning point the identity $\varphi_{00}^v =$ Id$_{\Omega}$ and terminating at the endpoint $\varphi_{01}^v = \varphi$ matching the given images.
Let $V$ denote the admissible Hilbert space of vector fields $v: \Omega \rightarrow R^{d}$, included in $L^2(\Omega, R^d)$ and associated with norm $\Vert \cdot \Vert_V$, let $\mathcal{X}_{V}^{2}$ denote the set of time-dependent vector field
such that for each $t$, $v(t)\in V$ and 
\begin{displaymath}
\Vert v \Vert_{\mathcal{X}_{V}^{2}}^2 =\int_0^1 \Vert v(t) \Vert_V^2 dt < \infty.
\end{displaymath}
The set of the flow at time $1$ builds a group of diffeomorphisms with respect to the composition of functions denoted by $G_V = \{\varphi_{01}^v, v \in \mathcal{X}_{V}^{2} \}$. Moreover, $G_V$ can be given structure as an infinite-dimensional manifold and hence it is also a Lie group. $T_\varphi G_V$ denotes the tangent space at the point $\varphi$. 

The length of the curve connecting the initial point $\varphi_{00}^v$ to the final point $\varphi_{01}^v$ on the group of diffeomorphisms $G_V$ is expressed as
\begin{equation}
    L(\varphi) = \int_0^1 \langle \dot{\varphi}(t), \dot{\varphi}(t) \rangle_{T_{\varphi}G_V}^{\frac{1}{2}} dt
\end{equation}
with regard to an inner product $\langle \cdot , \cdot \rangle_{T_{\varphi}G_V}$ on each tangent space, where $\dot{\varphi}$ denotes differentiation with respect to time. The inner product on tangent space $T_\varphi G_V$ associated with the norm $\Vert \cdot \Vert_{T_\varphi G_V}$ makes $G_V$ a Riemannian manifold. The corresponding energy is
\begin{equation}
    E_R(\varphi) = \frac{1}{2}\int_0^1 \langle \dot{\varphi}(t), \dot{\varphi}(t) \rangle_{T_\varphi G_V} dt = \frac{1}{2}\int_0^1 \Vert \dot{\varphi}(t) \Vert_{T_\varphi G_V}^{2} dt.
\end{equation}
Choosing the metric to be right-invariant metric on $G_V$ so that $\Vert \dot{\varphi}(t) \Vert_{T_\varphi G_V} = \Vert v(t) \Vert_{T_{\text{Id}}G_V}$, we can write $E_R$ as
\begin{equation}
    E_R(\varphi) = \frac{1}{2}\int_0 ^ 1 \|v(t)\|_{T_{\text{Id}}G_V}^2 dt,
    \label{eq: regularizer}
\end{equation}
where $v(t)$ is defined in the differential equation \cref{eq: differential equation}.
In this framework, the admissible Hilbert space $V$ mentioned previously is equal to the tangent space of $G_V$ at the identity $\text{Id}$, with $V = T_{\text{Id}}G_V$ and $\Vert \cdot \Vert_{V} = \Vert \cdot \Vert_{T_{\text{Id}}G_V}$. Minimizing the functional in \cref{eq: functional} with the regularizer in \cref{eq: regularizer} yields a geodesic with the shortest length path in $G_V$. 
In the next step, we will describe how the admissible Hilbert space $V$ can be generated. 

\subsection{Reproducing Kernel and Momentum}
\label{subsec: Reproducing Kernel and Momentum}
One way to construct the space $V$ is to use an inner product defined through a differential operator $L$ given by 
\begin{equation}
    \langle u , v \rangle_V := \langle Lu , v \rangle _{L^2} = \langle u , Lv \rangle_{L^2},
\label{eq: differential operator}
\end{equation}
where $\langle \cdot , \cdot \rangle _{L^2}$ is the usual $L^2$-product for square integrable vector fields on $\Omega$.
The induced norm is 
\begin{equation*}
    \Vert v\Vert_V := \sqrt{\langle v , v \rangle_V}.
\end{equation*}
In fact, the operator $L$ is a duality operator between $V$ and its dual space $V^{\ast}$, $L: V\rightarrow V^{\ast}$ is also referred to as the momentum operator.
This is due to the fact that the momentum operator connects the inner product on $V$ to the inner product in $L^2$ as can be seen in \cref{eq: differential operator}, and the image $Lv$ of an element $v\in V$ is referred to as the momentum of $v$, i.e. $m = Lv$.

Each point $x\in \Omega$ specifies a linear evaluation functional Diracs $\delta_x$ defined by $(\delta_x|v) = v(x)$ for $v \in V$, this means that $\delta_x \in V^{\ast}$.
This implies that for any $a\in R^d$, the function $a\otimes \delta_x: v\mapsto a^T v(x)$ is also a continuous linear functional on $V$, so it belongs to $V^{\ast}$. According to the Riesz representation theorem, for every $v \in V$, there exists a reproducing kernel $K: \Omega \times \Omega \rightarrow R^{d \times d}$ such that
\begin{equation}
   (a\otimes \delta_x |v )=  \langle K (., x)a, v \rangle_V = a^T v(x).
\end{equation}
This implies that
\begin{equation}
   \langle LK (., x)a, v \rangle_{L^2} = \langle K (., x)a, v \rangle_V =  (a\otimes \delta_x |v ) =  a^T v(x) = \langle a\otimes \delta_x, v \rangle_{L^2}.
\end{equation}
Thus $LK (., x)a = a\otimes \delta_x$, so we can view $K$ as an inverse of $L$, i.e. the inverse duality operator of $V$, and $K$ is often viewed as a convolution and thus $v = K \ast m$. In this way, the differential operator $L$ is used to construct the kernel $K$ and the space $V$.

\subsection{The EPDiff Equation}
\label{subsec: The EPDiff Equation}
For LDDMM image registration, to get an optimal path over $G_V$, the energy functional which is minimized over $\mathcal{X}_{V}^{2}$ takes the form
\begin{equation}
    E(v) = E_S(\varphi\cdot I_0, I_1)+ \frac{1}{2}\int_0 ^ 1 \|v(t)\|_{V}^2 dt.
    \label{eq: LDDMM energy}
\end{equation}
By using the calculus of variations, the geodesic equations can be retrieved. There are Euler-Poincar\'{e} equations on the diffeomorphism group and they are often denoted EPDiff:
\begin{align}
        &\frac{\partial m(t)}{\partial t} + v(t) \cdot \nabla m(t) + (\nabla v(t))^{T} \cdot m(t) + m(t) \nabla \cdot v(t) = 0, \label{eq: EPDiff}\\    
        &v(t) = K \ast m(t).    \label{eq: connect}
\end{align}
Given an optimal initial momentum $m(0)$, according to the time evolution of momentum \cref{eq: EPDiff} and the connection between momentum and velocity \cref{eq: connect}, 
the entire path of velocities $v(t)$ can be recovered, and the corresponding optimal deformation $\varphi$ also can be reconstructed via equation \cref{eq: differential equation}.

\section{Registration with Higher-order Kernels and Momenta}
\label{sec: Registration with Higher-order Kernels and Momenta}
The LDDMM framework only allows for local smooth translations using zeroth-order momentum, as demonstrated in \cref{fig: Gaussian kernel zeroth simulation} in \cref{sec: Introduction}. If we desire to create more general local deformation, higher-order momentum can be considered, as suggested by \cite{sommer_higher-order_2013}. 
Additionally, the use of non-differentiable function and its corresponding first-order momentum, as seen in \cref{fig: Multiplicative Wendland kernel first simulation}, enables the simulation of sliding motion. 
Motivated by these observations, in this section, we will present a sliding motion registration method that incorporates non-differentiable kernel and its corresponding first-order momentum within the LDDMM framework.
We first introduce the deformation encoded by higher-order kernels and momenta, then give the evolution equations and finally we introduce a new type of kernel that is based on non-differentiable function.

\subsection{High-order Kernels and Momenta}
\label{subsec: High-order Kernels and Momenta}
Inspired by previous work \cite{sommer_higher-order_2013,jacobs_higher-order_2014}, this paper aims to obtain a more comprehensive representation of the local deformation by incorporating partial derivatives of the kernel into the representation.
According to \cite{zhou_derivative_2008}, the reproducing property of $V$ also holds true for partial derivatives of the kernel.  
For any vector $v\in V$, its partial derivative $D^{\alpha}v $ at $x\in \Omega$ is defined as
\begin{equation}
    D^{\alpha}v(x) = D_x^\alpha v =\frac{\partial ^{|\alpha|}}{\partial(x^1)^{\alpha^1}\cdots\partial(x^d)^{\alpha^d} }v(x),
\end{equation}
where $\alpha = (\alpha^1, \cdots, \alpha^d)$ and $|\alpha| =\sum_{i=1}^d \alpha^j$. 
When using the definition of higher-order Diracs $a\otimes D_x^\alpha$ in \cite{sommer_higher-order_2013}, the partial derivative reproducing property follows that
\begin{equation}
   ( a\otimes D_x^\alpha |v) = \langle D^\alpha K (., x)a, v \rangle_V = a^T D_x^\alpha v.
\end{equation}
This implies
\begin{equation}
   \langle LD^\alpha K (., x)a, v \rangle_{L^2} = \langle D^\alpha K (., x)a, v \rangle_V =   ( a\otimes D_x^\alpha |v) =  a^T D_x^\alpha v = \langle a\otimes D_x^\alpha, v \rangle_{L^2},
\end{equation}
thus $LD^\alpha K (., x)a = a\otimes D_x^\alpha$.
As a consequence, the higher-order Diracs $a\otimes D_x^\alpha$ is connected to partial derivatives $D^\alpha K (., x)$ of the kernel.

In this paper, the focus is on the zeroth- and first-order partial derivatives of the kernel represented as
\begin{equation}
    v = K \ast m_0 + \sum_{i=1}^d D^i K \ast m_i,
    \label{eq: velocity decompose}
\end{equation}
where $D^i K$ is the partial derivative with respect to the $i$th coordinate and $m_i$ is the corresponding momentum for $i = 1, \cdots ,d$. 
Therefore equation \cref{eq: velocity decompose} extends the only zeroth-order momentum to a linear combination of zeroth- and first-order momenta. 
\begin{remark}
     Below, when modelling sliding motion along a boundary, we are generally only interested in the first-order directional derivative kernel along the unit tangent vector $w$ of the boundaries at point $\textbf{y}$. In this case, the linear combination of partial derivative kernels \cref{eq: velocity decompose} is equivalent to using the directional derivative kernel $\nabla_{w(\textbf{y})}K$. We generally use \cref{eq: velocity decompose} because it frees us from needing prior knowledge of the unit tangent vector $w(\textbf{y})$. 
\end{remark}

For the sake of simplicity, we denote $v_0 =  K \ast m_0 $, $v_i=D^i K \ast m_i, i = 1, \cdots,d$, then it is seen that the equation \cref{eq: velocity decompose} offers a decomposition of the velocity by 
\begin{equation}
    v = \sum_{i=0}^d v_i.
\end{equation} 
This extension allows for the velocity to be decomposed into a linear combination of zeroth- and first-order momenta, offering a multiscale structure for analyzing velocity at different scales. 
The regularization energy can be extended to:
\begin{equation}
    E_R(\varphi) = \frac{1}{2}\int_0 ^ 1 \sum_{i=0}^d \|v_i(t)\|_V^2 dt,
\end{equation}
with
\begin{equation}
   \sum_{i=0}^d \|v_i(t)\|_V^2
   = \left \langle m_0(t), v_0(t)\right \rangle_{L^2} + \sum_{i=1}^d \left \langle m_i(t), D^i v_i(t)\right \rangle_{L^2}.
\end{equation}
The analysis of extremal equations for the energy can still be performed since the partial derivatives of kernels are also members of $V$ and the first-order momentum similarly belongs to the dual space $V^{\ast}$.

\subsection{Euler-Lagrange equations}
\label{subsec: Euler-Lagrange equations}
We now detail the computation of the gradient for the new energy:
\begin{equation}
    E(v) =  E_S(\varphi_{01}^{v})+ \frac{1}{2}\int_0 ^ 1 \sum_{i=0}^d \|v_i(t)\|_V^2 dt,
    \label{eq: HOLDDMM energy}
\end{equation}
with $v(t) = \sum_{i=0}^d v_i(t) \in V$.
Considering a variation $h(t) = \sum_{i=0}^d h_i(t) \in V$ and calculate 
\begin{equation}
    \frac{d}{d\epsilon}E(\varphi_{01}^{v+\epsilon h})|_{\epsilon=0}  =  \frac{d}{d\epsilon}E_S(\varphi_{01}^{v+\epsilon h})|_{\epsilon=0} + \int_0^1 \left \langle v(t), h(t) \right \rangle_V dt. 
\end{equation}
Following \cite{younes_shapes_2019}, we can define adjoint operator $\text{Ad}_\varphi v$ and its conjugation $\text{Ad}_{\varphi}^{\ast}\rho$ for $v\in V$ and $\rho \in V^{\ast}$ as:
$\text{Ad}_\varphi v(x) = (D\varphi v)\circ \varphi^{-1}(x)$, $(\text{Ad}_{\varphi}^{\ast}\rho | v) = (\rho|\text{Ad}_\varphi v)$, we can also define 
$\text{Ad}_{\varphi}^{T}v = K(\text{Ad}_{\varphi}^{\ast}Lv)$, which then satisfies $\left \langle \text{Ad}_{\varphi}^{T}v, w\right \rangle_V = (\text{Ad}_{\varphi}^{\ast}Lv| w) = (Lv|\text{Ad}_{\varphi}w)$.
Let $\overline{\partial}E_S(\varphi_{01}^{v})$ denote the Eulerian differential of $E_S$ at $\varphi_{01}^{v}$ and it belongs to $V^{\ast}$, the corresponding V-Eulerian gradient at $\varphi_{01}^{v}$ is denoted as $\overline{\nabla}^V E_S(\varphi_{01}^{v}) \in V$.
It is shown in \cite{younes_shapes_2019} that 
\begin{equation}
    \frac{d}{d\epsilon}\varphi_{01}^{v+\epsilon h}(x)|_{\epsilon=0} = \int_0^1 (\text{Ad}_{\varphi_{t1}^v}h(t))\circ \varphi_{01}^v(x) dt.
\end{equation}
Letting $w_i = \int_0^1 \text{Ad}_{\varphi_{t1}^v}h_i(t)dt$ and $w = \sum_{i=0}^d w_i$, now let's calculate
$\frac{d}{d\epsilon}E_S(\varphi_{01}^{v+\epsilon h})|_{\epsilon=0} $:
\begin{align}
    \frac{d}{d\epsilon}E_S(\varphi_{01}^{v+\epsilon h})|_{\epsilon=0} 
    &= \partial_{\epsilon}E_S(\varphi_{0\epsilon}^{w} \circ \varphi_{01}^{v})|_{\epsilon=0}\nonumber\\
    &= \big(\overline{\partial}E_S(\varphi_{01}^{v})|w\big) \nonumber\\
    &= \sum_{i=0}^d \Big(\overline{\partial}E_S(\varphi_{01}^{v})| w_i\Big) \nonumber\\
    &= \sum_{i=0}^d \int_0^1 \Big(\overline{\partial}E_S(\varphi_{01}^{v})| \text{Ad}_{\varphi_{t1}^v}h_i(t)\Big) dt\nonumber\\
    &=  \int_0^1 \Big( \text{Ad}_{\varphi_{t1}^v}^{\ast}\overline{\partial}E_S(\varphi_{01}^{v})|h(t)\Big) dt\nonumber\\
    &= \int_0^1 \left \langle  \text{Ad}_{\varphi_{t1}^v}^{T}\overline{\nabla}^V E_S(\varphi_{01}^{v}) , h(t) \right \rangle_V dt.
\end{align}
From above equation we can derive the $\mathcal{X}_{V}^{2}$ gradient of $E_S$, i.e. the Eulerian gradient of $E_S$ at different time $\nabla E_S(v)(t) = \text{Ad}_{\varphi_{t1}^v}^{T}\overline{\nabla}^V E_S(\varphi_{01}^{v})$.
 So if $v \in \mathcal{X}_{V}^{2}$ is a minimizer for \cref{eq: HOLDDMM energy}, then for all $t$, we have
 \begin{equation}
 \label{eq: first-order velocity evolution}
     v(t) +   \text{Ad}_{\varphi_{t1}^v}^{T}\overline{\nabla}E_S(\varphi_{01}^{v}) = 0,
 \end{equation}
then we can get $v(t) = \text{Ad}_{\varphi_{t1}^v}^{T}v(1)$, where $ v(1)= -\overline{\nabla}E_S(\varphi_{01}^{v})$, further more we have $v(t) =  \text{Ad}_{\varphi_{t0}^v}^{T}v(0)$. Since $v(t) = \sum_{i=0}^d v_i(t)$, it also holds true for every $v_i(t)$: $v_i(t) =  \text{Ad}_{\varphi_{t0}^v}^{T}v_i(0), i =0,\cdots, d$.

If we denote the total momentum as $m(t) = \sum_{i=0}^d m_i(t)$, then we have 
\begin{align}
    (m(t)|w) 
             &=  \sum_{i=0}^d \left \langle v_i(t), w\right \rangle_V 
             = \sum_{i=0}^d \left \langle \text{Ad}_{\varphi_{t0}^v}^{T}v_i(0), w\right \rangle_V 
             = \sum_{i=0}^d \left \langle v_i(0), \text{Ad}_{\varphi_{t0}^v}w\right \rangle_V  \nonumber\\
             &= \sum_{i=0}^d (m_i(0)|\text{Ad}_{\varphi_{t0}^v}w) 
             = (m(0)|\text{Ad}_{\varphi_{t0}^v}w),
\end{align}
then
\begin{align}
    \partial_t(m(t)|w) &= (m(0)|\partial_t\text{Ad}_{\varphi_{t0}^v}w) 
    = -(m(0)|\text{Ad}_{\varphi_{t0}^v}\text{ad}_{v(t)}w)\nonumber\\
    &= -(m(t)|\text{ad}_{v(t)}w)
    = -(\text{ad}_{v(t)}^{\ast} m(t)|w),
\end{align}
so for total momentum, we have
\begin{equation}
 \label{eq: first-order momenta evolution}
    \partial_t m(t) = - \text{ad}_{v(t)}^{\ast} m(t),
\end{equation}
and for each momentum
\begin{equation}
    \partial_t m_i(t) = - \text{ad}_{v(t)}^{\ast} m_i(t), i= 0, \cdots, d.
\end{equation}

\subsection{Compactly Supported Reproducing Kernels}
\label{subsec: Compactly Supported Reproducing Kernels}
An alternative approach to constructing an admissible Hilbert space is to use reproducing kernels directly instead of deriving it from a differential operator, as introduced in the previous section. 
One commonly used kernel is the Gaussian kernel $K_g(\textbf{x}, \textbf{y}) = \text{exp}(-\frac{\|\textbf{x} - \textbf{y}\|^2}{\sigma_g^2})$, $\sigma_g > 0$. Since this kind of kernel is infinitely supported, the computation is computationally expensive. 
It is therefore essential that the applied kernel not only has compact support but also simultaneously has the reproducing property such that the constructed space is a reproducing kernel Hilbert space. 

For equation \cref{eq: velocity decompose}, the velocity field can be represented as a combination of basis functions. When the velocity field is smooth, using smooth kernels is effective. However, in cases of discontinuous velocity fields, such as sliding motion, the use of smooth kernels cannot produce the desired results. Therefore, to deal with sliding motion, we shift our focus towards non-differentiable kernels, which have been demonstrated to handle discontinuous deformation \cite{jud_sparse_2016}.
Our approach takes inspiration from \cite{pai_kernel_2016} that utilized compactly supported reproducing Wendland kernels to parameterize velocity fields,
but we concentrate on non-stationary velocity fields rather than the stationary velocity fields considered in \cite{pai_kernel_2016}.


The zeroth kind Wendland kernel is a viable option for addressing discontinuous deformation, as it not only has the reproducing property but also has compact support for efficient computation. Additionally, its non-differentiability attribute enables discontinuous deformation, making it a focus of this paper. The one-dimensional $C^0$ Wendland kernel takes the form
 \begin{equation}
  k_0 (x, y) = \left\{\left(1-\frac{\|x - y\|}{\sigma_0}\right)_+ \right\}^2,
\end{equation}
with $\sigma_0 > 0$, $a_+ = \text{max}(0, a)$. 
This kernel is not differentiable at $x = y$, as seen in \cref{fig:wendland2D}, and hence enables discontinuous deformation.

\begin{figure}[!htbp]
\setcounter{subfigure}{0} 
\centering
\subfigure[$C^0$ Wendland kernel]{\includegraphics[width=0.4\textwidth]{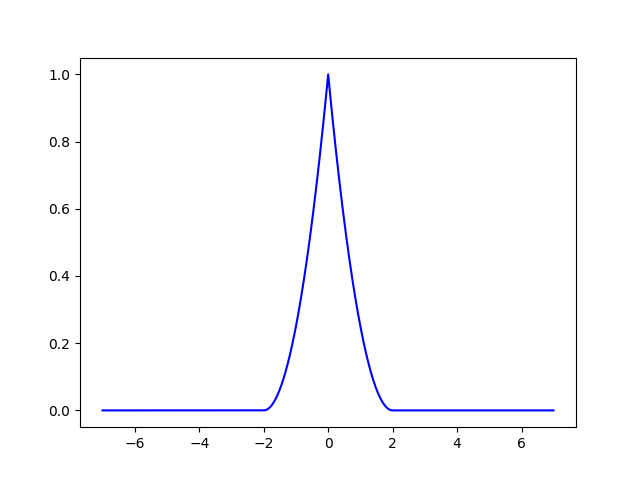}}
\subfigure[The derivative of $C^0$ Wendland kernel]{\includegraphics[width=0.4\textwidth]{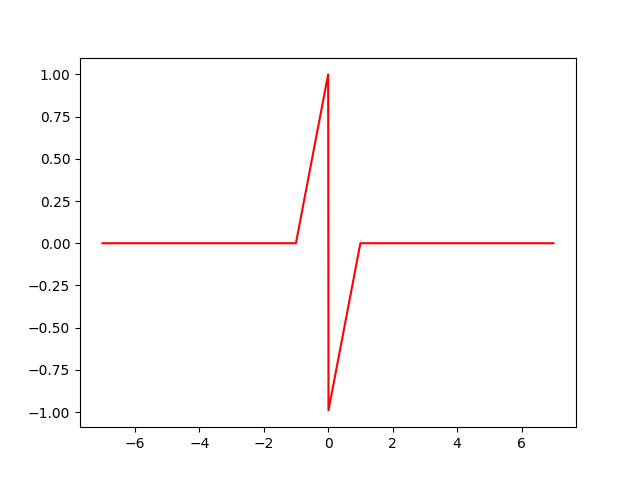}}
\caption{\label{fig:wendland2D}1-D $C^0$ Wendland kernel and its corresponding derivative kernel.}
\end{figure}

Because kernels are closed under product, i.e. the multiplication of valid kernels is also a valid kernel \cite{shawe-taylor_kernel_2004},
so for higher dimension $d$, 
we construct the kernel by multiplying one-dimensional $C^0$ Wendland kernels: 
\begin{equation}
  K_m (\textbf{x}, \textbf{y}) = \prod_{i=1}^{d} k_0 (x^i, y^i),
\end{equation}
where $\textbf{x} = (x^1,\cdots, x^d)$ and $\textbf{y} = (y^1,\cdots, y^d)\in R^d$.
The constructed multiplicative kernel hence is also a valid kernel. In \cref{fig: kernel and its derivative}, the 2-D multiplicative $C^0$ Wendland kernel and its corresponding two derivative kernels are plotted, as can be seen that this multiplicative kernel $K_m$ is non-differentiable at $\textbf{x} = \textbf{y}$ and the derivative kernels has discontinuous planes.
\begin{figure}[htbp]
\centering
\subfigure[$K_m$]{
\includegraphics[width=0.3\textwidth]{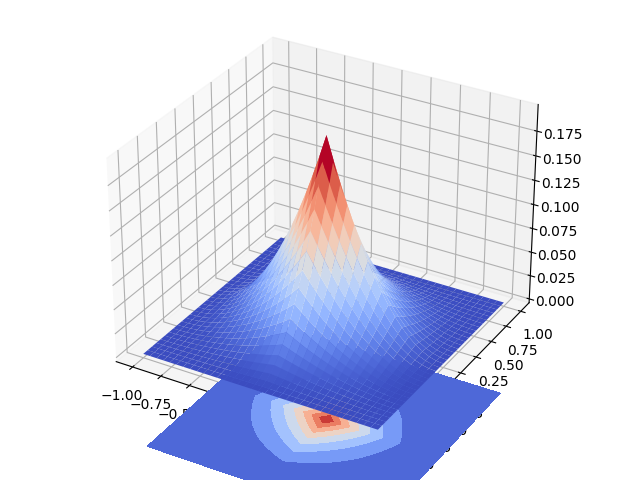}
}
\centering
\subfigure[$\nabla_x K_m$]{
\includegraphics[width=0.3\textwidth]{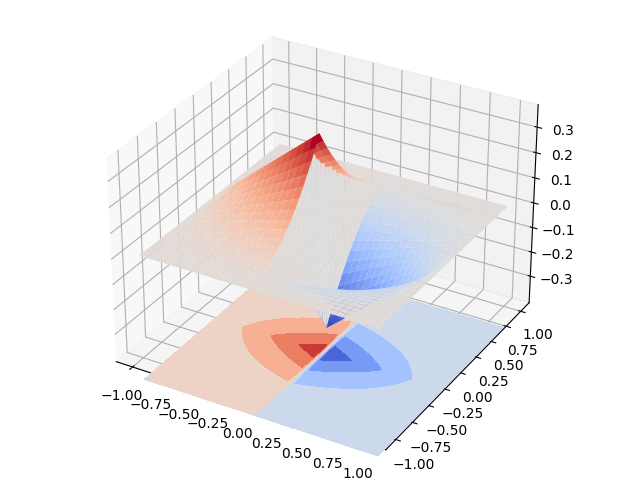}
}
\centering
\subfigure[$\nabla_y K_m$]{
\includegraphics[width=0.3\textwidth]{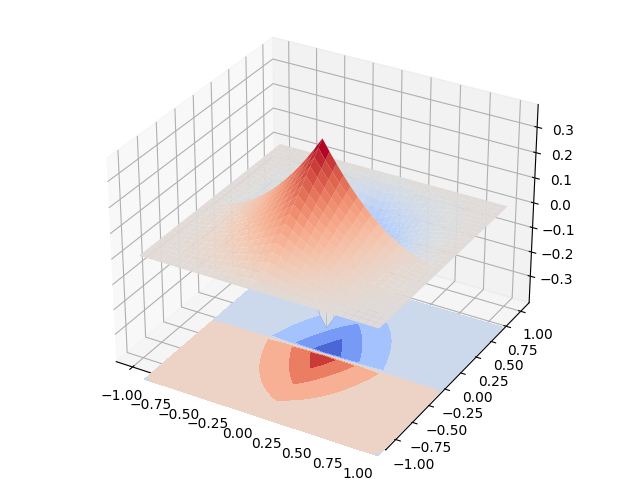}
}
\caption{\label{fig: kernel and its derivative}2-D multiplicative $C^0$ Wendland kernel and its corresponding two derivative kernels.}
\end{figure} 




To fully capture sliding motion, it is necessary to allow the momenta to exhibit sparsity across orders, allowing them to be zero under some orders and non-zero under others at the same spatial locations. 
To achieve this, a sparsity constraint is required.
The $L^1$-norm is commonly used as a sparsity penalty, so it is employed here as the sparsity prior. The incorporation of the sparsity term extends the registration functional \cref{eq: HOLDDMM energy} to:
\begin{equation}
    E(\varphi) =  E_S(\varphi)+  E_R(\varphi) + \sum_{i=0}^{d}\lambda_i\|m_i(0)\| ,
    \label{eq: sparse HOLDDMM energy}
\end{equation}
where $\lambda_i$ represents the sparsity weight on each momentum.


\section{Discontinuous velocity field}
\label{sec: Discontinuous velocity field}
The use of the non-smooth multiplicative Wendland kernel and its first-order momentum can result in a discontinuous velocity field, then causing a differential equation with discontinuous right-hand side. 
In this section, we will analyze the flow generated by this velocity field using the theories of non-smooth mechanical systems  \cite{leine_dynamics_2004, dieci_fundamental_2011} based on fundamental solution matrix and saltation matrix. The definition and discussion of the fundamental solution matrix and saltation matrix are given in \cref{subsec: Fundamental Solution Matrix} and \cref{subsec: Saltation matrix} respectively. Then the evolution equation for the non-smooth case is presented in \cref{subsec: Differentiation with Respect to the Vector Field}.

\subsection{Fundamental Solution Matrix}
\label{subsec: Fundamental Solution Matrix}
Let us consider a general differential system:
\begin{equation}
\label{eq: differential system}
    \frac{\partial \varphi(t, \textbf{x})}{\partial t} = v(t, \varphi(t, \textbf{x})), \varphi(0, \textbf{x}) = \textbf{x}, \forall~\textbf{x} \in \Omega.
\end{equation}
Given a reference initial point $\textbf{x}_0$ at time zero, $\varphi_{0t}^v(\textbf{x}_0)$ is the position of $\textbf{x}_0$ at time $t$. We are interested in the time evolution of the difference $\Delta \textbf{x}(t)$ between the reference trajectory $\varphi_{0t}^v(\textbf{x}_0)$ and the trajectory $\varphi_{0t}^v(\textbf{y}_0)$ with initial point $\textbf{y}_0$, which is any possible infinitesimal perturbation of $\textbf{x}_0$:

\begin{equation} 
\label{eq: evolution of the difference}
\Delta \textbf{x}(t)=\varphi_{0t}^v(\textbf{y}_0) -\varphi_{0t}^v(\textbf{x}_0) = D\varphi_{0t}^v(\textbf{x}_0)(\textbf{y}_0 - \textbf{x}_0) + O(\|\textbf{y}_0 - \textbf{x}_0\|^2),
\end{equation}
where $\textbf{y}_0 = \textbf{x}_0 + \Delta \textbf{x}_0$, $\Delta \textbf{x}_0$ represents the infinitesimal disturbance of $\textbf{x}_0$.

In order to analyze the flow, we first give a definition of the fundamental solution matrix as follows based on \cite{bizzarri_necessary_2016}.
\begin{definition}[Fundamental solution matrix]
\label{def: Fundamental solution matrix}
A fundamental solution matrix is the first-order expansion of the flow of perturbations around the reference trajectory.
\end{definition}
According to the above definition, the matrix $D\varphi_{0t}^v(\textbf{x}_0)$ in \cref{eq: evolution of the difference} is the defined fundamental solution matrix along $\varphi_{0t}^v(\textbf{x}_0)$ to transit the initial difference $\Delta \textbf{x}_0$ to latter difference at time $t$, so it is also known as the state transition matrix.
The following proposition is quite important for our analysis.
\begin{proposition}[Transition property]
For $0<t_1<t$, the flow $\varphi_{0t}^v(\textbf{x}_0)$ starts from the initial point $\textbf{x}_0$ at time $t_0$ and arrives at the point $\textbf{x}_1$ in time $t_1$, then the fundamental solution matrix satisfies
\begin{equation}
\label{eq: Transition property}
D\varphi_{0t}^v(\textbf{x}_0) = D\varphi_{t_1t}^v(\textbf{x}_1) D\varphi_{0t_1}^v(\textbf{x}_0).
\end{equation}
\end{proposition}

 \begin{proof}
 Using \cref{eq: evolution of the difference}, we have
 \begin{align*}
 \Delta \textbf{x}(t_1)&= D\varphi_{0t_1}^v(\textbf{x}_0)\Delta \textbf{x}_0 + O(\|\Delta \textbf{x}_0\|^2),\\
\Delta \textbf{x}(t)&= D\varphi_{t_1t}^v(\textbf{x}_1)\Delta \textbf{x}(t_1) + O(\|\Delta \textbf{x}(t_1)\|^2),
 \end{align*}
 so we will get
 \begin{align*}
 \Delta \textbf{x}(t)
 &= D\varphi_{t_1t}^v(\textbf{x}_1)(D\varphi_{0t_1}^v(\textbf{x}_0)\Delta \textbf{x}_0 + O(\|\Delta \textbf{x}_0\|^2)) + O(\|\Delta \textbf{x}(t_1)\|^2),\\
&= D\varphi_{t_1t}^v(\textbf{x}_1)D\varphi_{0t_1}^v(\textbf{x}_0)\Delta \textbf{x}_0 + \text{h.o.t},\\
&=D\varphi_{0t}^v(\textbf{x}_0)\Delta \textbf{x}_0 + \text{h.o.t},
 \end{align*}
 where "h.o.t" stands for higher-order terms.
 Thus for $0<t_1<t$, we have 
 \begin{displaymath}
 D\varphi_{0t}^v(\textbf{x}_0) = D\varphi_{t_1t}^v(\textbf{x}_1) D\varphi_{0t_1}^v(\textbf{x}_0).
 \end{displaymath}
 \end{proof}

Due to the incorporation of first-order momenta based on non-smooth function, the velocity field generated by our proposed method could be discontinuous. The differential system \cref{eq: differential system} stemming from this velocity field may come with discontinuous right-hand side.
For a point $\textbf{x}$, if the corresponding first-order momenta at this point is non-zero, then they will cause discontinuous velocity at a hyper-surface $\Sigma$ passing through $\textbf{x}$. As the velocity fields switch at $\Sigma$, the hyper-surface $\Sigma$ is also called the switching boundary. The non-zero first-order momenta determine the location of the switching boundary, and since we have imposed a sparsity constraint on the momenta across the order, we can assume that the switching boundaries are finite.
The switching boundary $\Sigma$ can be represented as the zero level set of a function $H: \Omega\rightarrow R$, if the point $\textbf{x}$ is on $\Sigma$, then $H(\textbf{x}) = 0 \Longleftrightarrow  \textbf{x} \in \Sigma$.


If the trajectory $\varphi_{0t}^v(\textbf{x}_0)$ never hits any switching boundary $\Sigma$ then it's a smooth case, the fundamental solution matrix $D\varphi_{0t}^v(\textbf{x}_0)$ is nothing but the time-dependent Jacobian matrix which can be obtained from the following initial value problem:
\begin{displaymath}
    \begin{cases}
    \frac{\partial D\varphi_{0t}^{v}(\textbf{x}_0)}{\partial t} =  Dv(t,\varphi_{0t}^{v}(\textbf{x}_0))D\varphi_{0t}^{v}(\textbf{x}_0)\\ D\varphi_{00}^{v}(\textbf{x}_0) = I.
    \end{cases}
\end{displaymath}
By integrating we can obtain $D\varphi_{0t}^{v}(\textbf{x}_0)= I+\int_0^t Dv(u,\varphi_{0u}^{v}(\textbf{x}_0))D\varphi_{0u}^{v}(\textbf{x}_0) du$.

The interesting case is when the trajectory hits switching boundaries leading to a jump in the fundamental solution matrix. Firstly we consider the case when it hits a single switching boundary $\Sigma_1$ only once at point $\textbf{x}_1$ in time $t_1$, then there is a jump behavior at $t_1$ for the fundamental matrix solution. Assume that the jump can be expressed with a matrix $S$, which maps the fundamental solution matrix before jump $D\varphi_{0t_1^-}^v(\textbf{x}_0)$ to the fundamental solution matrix after jump $D\varphi_{0t_1^+}^v(\textbf{x}_0)$:
\begin{equation}
\label{eq: saltation matrix}
    D\varphi_{0t_1^+}^v(\textbf{x}_0)=SD\varphi_{0t_1^-}^v(\textbf{x}_0),
\end{equation}
where $D\varphi_{0t_1^{\pm}}^v(\textbf{x}_0) = \lim_{t\to t_1^{\pm}}D\varphi_{0t_1}^v(\textbf{x}_0)$.

\begin{definition}[Saltation matrix]
    The matrix $S$ in \cref{eq: saltation matrix} is called saltation or jump matrix, describing the jump between $D\varphi_{0t_1^-}^v(\textbf{x}_0)$ and $D\varphi_{0t_1^+}^v(\textbf{x}_0)$.
\end{definition}
The saltation matrix $S$ can also be regarded as a fundamental solution matrix from time $t_1^-$ to $t_1^+$
\begin{displaymath}
S = D\varphi_{t_1^-t_1^+}^v(\textbf{x}_1).
\end{displaymath}

Now by means of the saltation matrix and transition property of the fundamental solution matrix, we can construct $D\varphi_{0t}^v(\textbf{x}_0)$ for $t>t_1$ as
\begin{align}
    \label{eq: product fundamental solution matrix}
    D\varphi_{0t}^v(\textbf{x}_0) &= D\varphi_{t_1^+t}^v(\textbf{x}_1) D\varphi_{0t_1^+}^v(\textbf{x}_0) \nonumber\\ 
    &=D\varphi_{t_1^+t}^v(\textbf{x}_1) SD\varphi_{0t_1^-}^v(\textbf{x}_0).
\end{align}
For $D\varphi_{0t_1^-}^v(\textbf{x}_0)$, it satisfies the following initial value problem:
\begin{equation}
    \begin{cases}
        \frac{\partial D\varphi_{0t}^{v}(x)}{\partial t} = Dv(t,\varphi_{0t}^{v}(x))D\varphi_{0t}^{v}(x)\\ D\varphi_{00}^{v}(x) = I.
    \end{cases}\label{eq: initial one}
\end{equation}
By integrating we have
\begin{equation}
\label{eq: before jump}
    D\varphi_{0t_1^-}^v(\textbf{x}_0)=I+ \int_0^{t_1^-}Dv(u,\varphi_{0u}^{v}(\textbf{x}_0))D\varphi_{0u}^{v}(\textbf{x}_0) du. 
\end{equation}     
For $D\varphi_{t_1^+t}^v(\textbf{x}_1)$, it satisfies the following initial value problem:
        \begin{equation}
        \begin{cases}
        \frac{\partial D\varphi_{t_1t}^{v}(\textbf{x}_1)}{\partial t} = Dv(t,\varphi_{t_1t}^{v}(\textbf{x}_1))D\varphi_{t_1t}^{v}(\textbf{x}_1)\\ D\varphi_{t_1t_1}^{v}(\textbf{x}_1) = I,
        \end{cases}\label{eq: initial two}
        \end{equation}
we also get
\begin{equation}
\label{eq: after jump}
    D\varphi_{t_1^+t}^v(\textbf{x}_1)=I+ \int_{t_1^+}^tDv(u,\varphi_{t_1^+u}^{v}(\textbf{x}_1))D\varphi_{t_1^+u}^{v}(\textbf{x}_1) du.
\end{equation}
Combining \cref{eq: product fundamental solution matrix} with \cref{eq: before jump,eq: after jump}, we can obtain the expression for the fundamental solution matrix $D\varphi_{0t}^v(\textbf{x}_0)$ for the discontinuous case when $t>t_1$.

Now thinking about the more complicated case that the flow hits boundaries at multiple but finite time points or switching boundaries. We start with two time points, assume the flow $\varphi_{0t}^v(\textbf{x}_0)$ hits switching boundaries at $(t_1, \textbf{x}_1)$ and $(t_2, \textbf{x}_2)$, $\textbf{x}_1 = \varphi_{0t_1}^v(\textbf{x}_0),\textbf{x}_2 = \varphi_{0t_2}^v(\textbf{x}_0)$. When $0<t_1<t_2<t$, using the transit property we can write the fundamental solution matrix $D\varphi_{0t}^v(\textbf{x}_0)$ as the product of different saltation matrices 
    \begin{align*}
    D\varphi_{0t}^v(\textbf{x}_0) 
    &=D\varphi_{t_2^+t}^v(\textbf{x}_2) S_2D\varphi_{t_1^+t_2^-}^v(\textbf{x}_1) S_1D\varphi_{0t_1^-}^v(\textbf{x}_0),
    \end{align*}
where the saltation matrices $S_1$ and $S_2$ are the description of the jump at time $t_1$ and $t_2$ respectively
\begin{displaymath}
S_1 = D\varphi_{t_1^-t_1^+}^v(\textbf{x}_1), S_2 = D\varphi_{t_2^-t_2^+}^v(\textbf{x}_2).
\end{displaymath}
For the fundamental solution matrices $D\varphi_{0t_1^-}^v(\textbf{x}_0),D\varphi_{t_1^+t_2^-}^v(\textbf{x}_1)$ and $D\varphi_{t_2^+t}^v(\textbf{x}_2)$, they all satisfy the corresponding initial value problems like \cref{eq: initial one,eq: initial two}. Combining all these expressions together, we can finally obtain the fundamental solution matrix 
$D\varphi_{0t}^v(\textbf{x}_0)$.

For more but finite time points, we can get an analogous expression of the fundamental solution matrix by multiplying one saltation matrix for each switching boundary. Since the fundamental solution matrix depends on the saltation matrix, next we will discuss the construction of the saltation matrix.


\subsection{Saltation matrix}
\label{subsec: Saltation matrix}
When a trajectory hits a switching boundary it can exhibit two distinct behaviors: crossing the boundary transversally or sliding along it. Different behaviors mean different forms for the saltation matrix $S$. 
In this section, we will give the expression of the saltation matrix based on \cite{leine_dynamics_2004}.

We first consider the case of transversal crossing switching boundaries. Assume $v^-$ and $v^+$ are the velocities of the flow $\varphi_{0t}^v(\textbf{x}_0)$ before and after hitting the switching boundary. For the switching boundary $\Sigma$, it is varying with the first-order momenta at different times, so we use $\Sigma(t)$ instead of $\Sigma$. We can also represent $\Sigma(t)$ as the zero level set of a function $H(t)$: 
if $\textbf{x} = \varphi_{0t}^v(\textbf{x}_0)\in \Sigma(t)$ if and only if $H(t,\textbf{x}) = 0$.

\begin{theorem}
    Assume the flow $\varphi_{0t}^v(\textbf{x}_0)$ hits the time-varying switching boundary $\Sigma(t)$ at point $\textbf{x}_1$ in time $t_1$ and then transversal cross it, $n(t_1, \textbf{x}_1)$ is the unit normal to $\Sigma(t_1)$ at $\textbf{x}_1$, then the saltation matrix has the form:
    \begin{equation}
   S = I + \frac{(v^+(t_1,\textbf{x}_1) - v^-(t_1,\textbf{x}_1))n^T(t_1, \textbf{x}_1)}{n^T(t_1, \textbf{x}_1) v^-(t_1,\textbf{x}_1)+\frac{\partial H(t_1, \textbf{x}_1)}{\partial t}}. 
   \label{eq: cross saltation}
\end{equation}
\end{theorem}

\begin{proof}
   Let $\textbf{y}_0$ be the infinitesimal perturbed initial condition of $\textbf{x}_0$: $\textbf{y}_0 = \textbf{x}_0 + \Delta \textbf{x}_0$. Suppose the trajectories of $\textbf{x}_0$ and $\textbf{y}_0$ hit the switching boundary $\Sigma(t)$ at $(t_1, \textbf{x}_1)$ and $(t_1+\Delta t,\textbf{y}_1)$, respectively (see \cref{fig:Disturbed and undisturbed trajectory }). 
 \begin{figure}[htbp]
\centering
\includegraphics[width=0.7\textwidth]{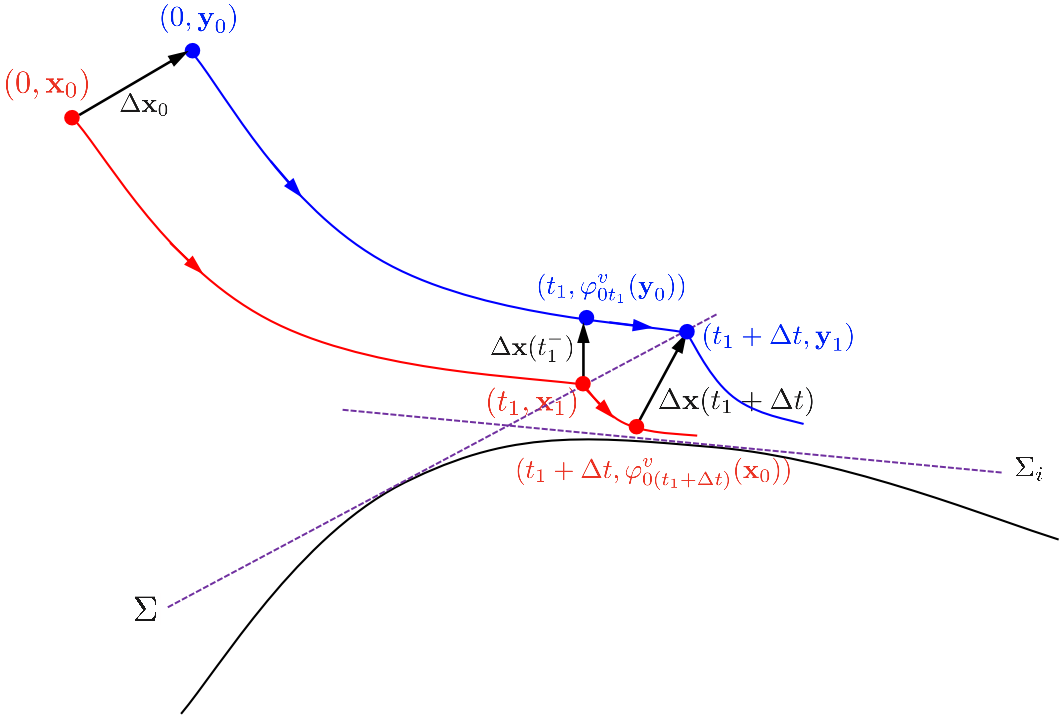}
\caption{\label{fig:Disturbed and undisturbed trajectory }Disturbed and undisturbed trajectory.}
\end{figure}

The difference between the disturbed and undisturbed solutions at time $t_1$ is denoted by
\begin{equation*}
\Delta \textbf{x}(t_1^-) = \varphi_{0t_1}^v(\textbf{y}_0) - \varphi_{0t_1}^v(\textbf{x}_0). 
\end{equation*}
The first-order Taylor expansions for the disturbed and undisturbed solutions at time $t_1+\Delta t$ are given by
\begin{align*}
\varphi_{0(t_1+\Delta t)}^v(\textbf{y}_0) 
&\approx \textbf{x}_1 + \Delta \textbf{x}(t_1^-)+ v^-(t_1,\textbf{x}_1)\cdot\Delta t + O(\Delta t^2), \\
\varphi_{0(t_1+\Delta t)}^v(\textbf{x}_0) &=\textbf{x}_1 + v^+(t_1,\textbf{x}_1)\cdot\Delta t,
\end{align*}
where $v^-(t_1,\textbf{x}_1)$ and $v^+(t_1,\textbf{x}_1)$ 
are the velocity of the flow $\varphi_{0t}^v(\textbf{x}_0)$ before and after hitting $\Sigma(t_1)$.
Then we can obtain the difference between the disturbed and undisturbed solutions at time $t_1+\Delta t$:
\begin{align}
\Delta \textbf{x}(t_1 + \Delta t) 
&\approx \Delta \textbf{x}(t_1^-)+ (v^-(t_1,\textbf{x}_1) - v^+(t_1,\textbf{x}_1))\Delta t.
\label{eq: disturb difference}
\end{align}
Then
\begin{align*}
H(t_1+\Delta t,\textbf{y}_1) 
&\approx H(t_1,\textbf{x}_1) + \frac{\partial H(t_1,\textbf{x}_1)}{\partial t}\cdot \Delta t +(\nabla H(t_1,\textbf{x}_1))^T\Big(\Delta \textbf{x}(t_1^-)+ v^-(t_1,\textbf{x}_1)\cdot\Delta t\Big) \\ 
&= (\nabla H(t_1,\textbf{x}_1))^T\Big(\Delta \textbf{x}(t_1^-)+ v^-(t_1,\textbf{x}_1)\cdot\Delta t\Big)+\frac{\partial H(t_1,\textbf{x}_1)}{\partial t}\cdot \Delta t  \\
&=0,  
\end{align*}
so we can get 
\begin{align}
\label{eq: time difference}  
  \Delta t &= -\frac{(\nabla H(t_1,\textbf{x}_1))^T \Delta \textbf{x}(t_1^-)}{(\nabla H(t_1,\textbf{x}_1))^T v^-(t_1,\textbf{x}_1)+\frac{\partial H(t_1, \textbf{x}_1)}{\partial t}} \nonumber \\
  &=-\frac{n^T(t_1, \textbf{x}_1) \Delta \textbf{x}(t_1^-)}{n^T(t_1, \textbf{x}_1) v^-(t_1,\textbf{x}_1)+\frac{\partial H(t_1, \textbf{x}_1)}{\partial t}},
\end{align}
where $n$ is the unit normal to $\Sigma(t_1)$ at $\textbf{x}_1$.
By inserting equation \cref{eq: time difference} into \cref{eq: disturb difference} we can get:
\begin{align*}
\Delta \textbf{x}(t_1 + \Delta t) 
&= \Delta \textbf{x}(t_1^-)+ \Big(v^-(t_1,\textbf{x}_1) - v^+(t_1,\textbf{x}_1)\Big)  \Big( -\frac{n^T(t_1, \textbf{x}_1) \Delta \textbf{x}(t_1^-)}{n^T(t_1, \textbf{x}_1) v^-(t_1,\textbf{x}_1)+\frac{\partial H(t_1, \textbf{x}_1)}{\partial t}}\Big)\\ 
&=\Big(I + \frac{(v^+(t_1,\textbf{x}_1) - v^-(t_1,\textbf{x}_1))n^T(t_1, \textbf{x}_1)}{n^T(t_1, \textbf{x}_1) v^-(t_1,\textbf{x}_1)+\frac{\partial H(t_1, \textbf{x}_1)}{\partial t}}\Big)\Delta \textbf{x}(t_1^-).
\end{align*}
Letting $\Delta t\to 0^+$, we can obtain $\Delta \textbf{x}(t_1^+) = S \Delta \textbf{x}(t_1^-)$, where
\begin{equation*}
   S = I + \frac{(v^+(t_1,\textbf{x}_1) - v^-(t_1,\textbf{x}_1))n^T(t_1, \textbf{x}_1)}{n^T(t_1, \textbf{x}_1) v^-(t_1,\textbf{x}_1)+\frac{\partial H(t_1, \textbf{x}_1)}{\partial t}}, 
\end{equation*} 
so that \cref{eq: cross saltation} follows.
\end{proof}

\begin{remark}
    When the trajectory hits $\Sigma(t_1)$ and does not cross it but slide along it, then the saltation matrix has another formulation: $S = I - n(t_1, \textbf{x}_1)n^T(t_1, \textbf{x}_1)$, the velocity after hitting $\Sigma(t_1)$ becomes $v^+(t_1,\textbf{x}_1) = Sv^-(t_1,\textbf{x}_1)$, which is tangent to the switching boundary leading to the sliding motion.
\end{remark}


After constructing the saltation matrix, the fundamental solution matrix $D\varphi_{0t}^v(\textbf{x}_0)$ and the adjoint operator are well-defined. We can then use them to derive the evolution equation for our non-smooth case.

\subsection{Differentiation with Respect to the Vector Field}
\label{subsec: Differentiation with Respect to the Vector Field}
We now study the differentiability of the flow generated by our proposed method with respect to the vector field. By leveraging the saltation matrix, we can redefine the adjoint operator $\text{Ad}_\varphi v$ using a uniform expression, analogous to the smooth case, with the fundamental solution matrix:
$\text{Ad}_\varphi v(\textbf{x}) = (D\varphi v)\circ \varphi^{-1}(\textbf{x})$.
Consequently, the differentiation with respect to $v$ also takes a uniform form as the smooth case, as we demonstrated in the following theorem.
\begin{theorem}
Let $v(t,\cdot), h(t,\cdot): \Omega \rightarrow R^{d}$ are the vector fields generated by first-order momenta based on non-differentiable function, then for $\textbf{x}_0 \in \Omega$, the generated flow is differentiable with respect to $v$ and
\begin{equation}
    \partial_v\varphi_{0t}^v h = \int_0^t (\text{Ad}_{\varphi_{ut}^v}h(u))\circ \varphi_{0t}^v(\textbf{x}_0) du.
\end{equation}
\end{theorem}

\begin{proof}
If the trajectory does not hit any switching boundaries, then we have a smooth case. According to \cite{younes_shapes_2019}, 
we obtain
\begin{equation*}
\partial_{\epsilon}\varphi_{0t}^{v+\epsilon h}|_{\epsilon = 0} = \partial_v\varphi_{0t}^v h =\int_0^t (D\varphi_{ut}^v  h(u))\circ \varphi_{0u}^v(\textbf{x}_0) du = \int_0^t (\text{Ad}_{\varphi_{ut}^v}h(u))\circ \varphi_{0t}^v(\textbf{x}_0) du.
\end{equation*}

The interesting case is when the flow $\varphi_{0t}^v(\textbf{x}_0)$ hits the switching boundary at multiple but finite time points or switching boundaries. 
Firstly we assume the flow hits the switching boundary only once at point $\textbf{x}_1$ in time $t_1$, then there is also a jump for the differentiation between time $t_1^-$ and $t_1^+$ expressed by the saltation matrix $S$:
\begin{displaymath}
\partial_v\varphi_{0t_1^+}^v(\textbf{x}_0) h=S\partial_v\varphi_{0t_1^-}^v(\textbf{x}_0) h.
\end{displaymath}
When $t>t_1$, by applying the chain rule, we can obtain
    \begin{align}
    \partial_v\varphi_{0t}^v(\textbf{x}_0) h 
    &=D\varphi_{t_1^+t}^v(\textbf{x}_1)  \partial_v\varphi_{0t_1^+}^v(\textbf{x}_0) h+\partial_v\varphi_{t_1^+t}^v(\textbf{x}_1)   h \nonumber \\ 
    &=D\varphi_{t_1^+t}^v(\textbf{x}_1)  S \partial_v\varphi_{0t_1^-}^v(\textbf{x}_0) h+ \partial_v\varphi_{t_1^+t}^v(\textbf{x}_1)   h .
    \label{eq: jump differentiation}
    \end{align}
Despite exhibiting a jump in the differentiation in \cref{eq: jump differentiation}, we can rewrite it in a uniform expression as the smooth case.
For flow $\varphi_{0t_1^-}^v(\textbf{x}_0)$, it satisfies 
    \begin{equation*}
         \partial_v\varphi_{0t_1^-}^v(\textbf{x}_0)  h = \int_0^{t_1^-} (\text{Ad}_{\varphi_{ut_1^-}^v}h(u))\circ \varphi_{0t_1^-}^v(\textbf{x}_0) du.
    \end{equation*}
For flow $\varphi_{t_1^+t}^v(\textbf{x}_1)$, it satisfies     
    \begin{equation*}
        \partial_v\varphi_{t_1^+t}^v(\textbf{x}_1)  h = \int_{t_1^+}^t (\text{Ad}_{\varphi_{ut}^v}h(u))\circ \varphi_{t_1^+t}^v(\textbf{x}_1) du,
    \end{equation*}
so for the whole flow, we have 
    \begin{align*}
    \partial_v\varphi_{0t}^v(\textbf{x}_0) h 
    =& D\varphi_{t_1^+t}^v(\textbf{x}_1) S \int_0^{t_1^-} (\text{Ad}_{\varphi_{ut_1^-}^v}h(u))\circ \varphi_{0t_1^-}^v(\textbf{x}_0) du 
    +\int_{t_1^+}^t (\text{Ad}_{\varphi_{ut}^v}h(u))\circ \varphi_{t_1^+t}^v(\textbf{x}_1) du\\ 
    =& \int_0^{t_1^-} (D\varphi_{t_1^+t}^v D\varphi_{t_1^-t_1^+}^v \text{Ad}_{\varphi_{ut_1^-}^v}h(u))\circ \varphi_{0t_1^-}^v(\textbf{x}_0) du 
    +\int_{t_1^+}^t (\text{Ad}_{\varphi_{ut}^v}h(u))\circ \varphi_{t_1^+t}^v(\textbf{x}_1) du\\ 
    =& \int_0^{t_1^-} (D\varphi_{t_1^+t}^v  D\varphi_{t_1^-t_1^+}^v \text{Ad}_{\varphi_{ut_1^-}^v}h(u))\circ (\varphi_{t_1^-t}^v)^{-1} \circ \varphi_{0t}^v(\textbf{x}_0) du \\
    &+\int_{t_1^+}^t (\text{Ad}_{\varphi_{ut}^v}h(u))\circ \varphi_{0t}^v(\textbf{x}_0) du\\
    =& \Big(\int_0^{t_1^-} (D\varphi_{t_1^+t}^v  D\varphi_{t_1^-t_1^+}^v  \text{Ad}_{\varphi_{ut_1^-}^v}h(u))\circ (\varphi_{t_1^-t}^v)^{-1} du 
    + \int_{t_1^+}^t \text{Ad}_{\varphi_{ut}^v}h(u) du \Big)\circ \varphi_{0t}^v(\textbf{x}_0) \\ 
    =& \Big(\int_{0}^{t_1^-} \text{Ad}_{\varphi_{ut}^v}h(u) du + \int_{t_1^+}^t \text{Ad}_{\varphi_{ut}^v}h(u) du \Big)\circ \varphi_{0t}^v(\textbf{x}_0)\\ 
    =& \Big(\int_{0}^{t} \text{Ad}_{\varphi_{ut}^v}h(u) du\Big) \circ \varphi_{0t}^v(\textbf{x}_0).
    \end{align*}   
For multiple but finite time points or switching boundaries, we can get an analogous uniform expression of $\partial_v\varphi_{0t}^v(\textbf{x}_0)h$. 
\end{proof}

\begin{corollary}
The evolution equations \cref{eq: first-order velocity evolution,eq: first-order momenta evolution} also apply to the non-smooth case. 
\end{corollary}
\begin{proof}
 The well-defined and uniform expression of the adjoint operator and differentiation with respect to the discontinuous vector field, which we obtained based on the saltation matrix and fundamental solution matrix, directly implies the result.
\end{proof}
As a result of the corollary, we can use the evolution equations to guide the motion of the trajectory generated by first-order momenta based on non-differentiable functions.

\section{Experiments}
\label{sec:Experiment}
We present experimental results for our proposed method for handling sliding motion using zeroth- and first-order momenta.
We show results for both 2D synthetic images and 3D lung volumes. Our 
implementation is available at \url{https://github.com/baolily/hokreg}, which is based on the $mermaid$ library (\url{https://github.com/uncbiag/mermaid}) that contains various image registration methods. 

To demonstrate the effectiveness of our proposed multiplicative Wendland kernel method, we include results obtained using LDDMM with Gaussian kernel as a contrast algorithm. Furthermore, to validate the advantage of the proposed first-order momentum strategy, we also compare our proposed method with a registration method that only uses zeroth-order momentum based on the multiplicative Wendland kernel. To ensure fair comparisons, 
we set the kernel size to 9 for different kernels and fix it for all experiments. 

\subsection{Synthetic Images}
\label{subsec: Synthetic Images}
The first evaluation was performed on a 2D synthetic rectangle image, where the original template image was warped by a sliding motion: the upper region was moved to the right while the lower region was moved to the left by 5 pixels, as shown in \cref{fig: rectangle template,fig: rectangle reference}.  
\begin{figure}[htbp]
\centering
\subfigure[]{
\includegraphics[height = 2.6cm]{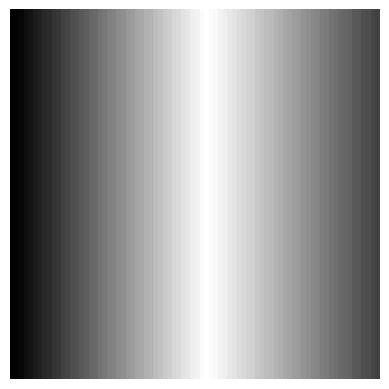}
\label{fig: rectangle template}}
\subfigure[]{
\includegraphics[height = 2.6cm]{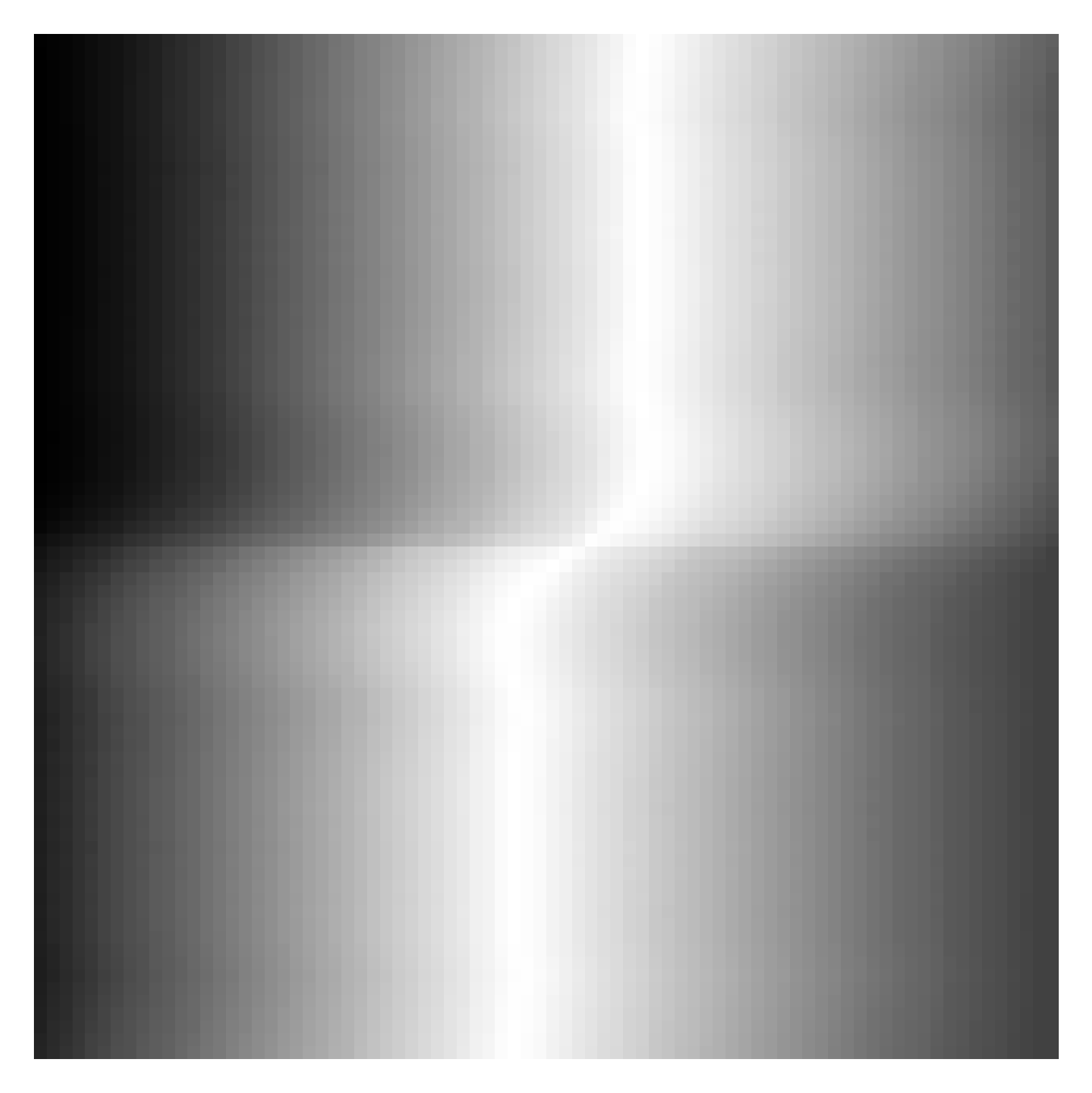}
\label{fig: rectangle deformed gaussian}}
\subfigure[]{
\includegraphics[height = 2.6cm]{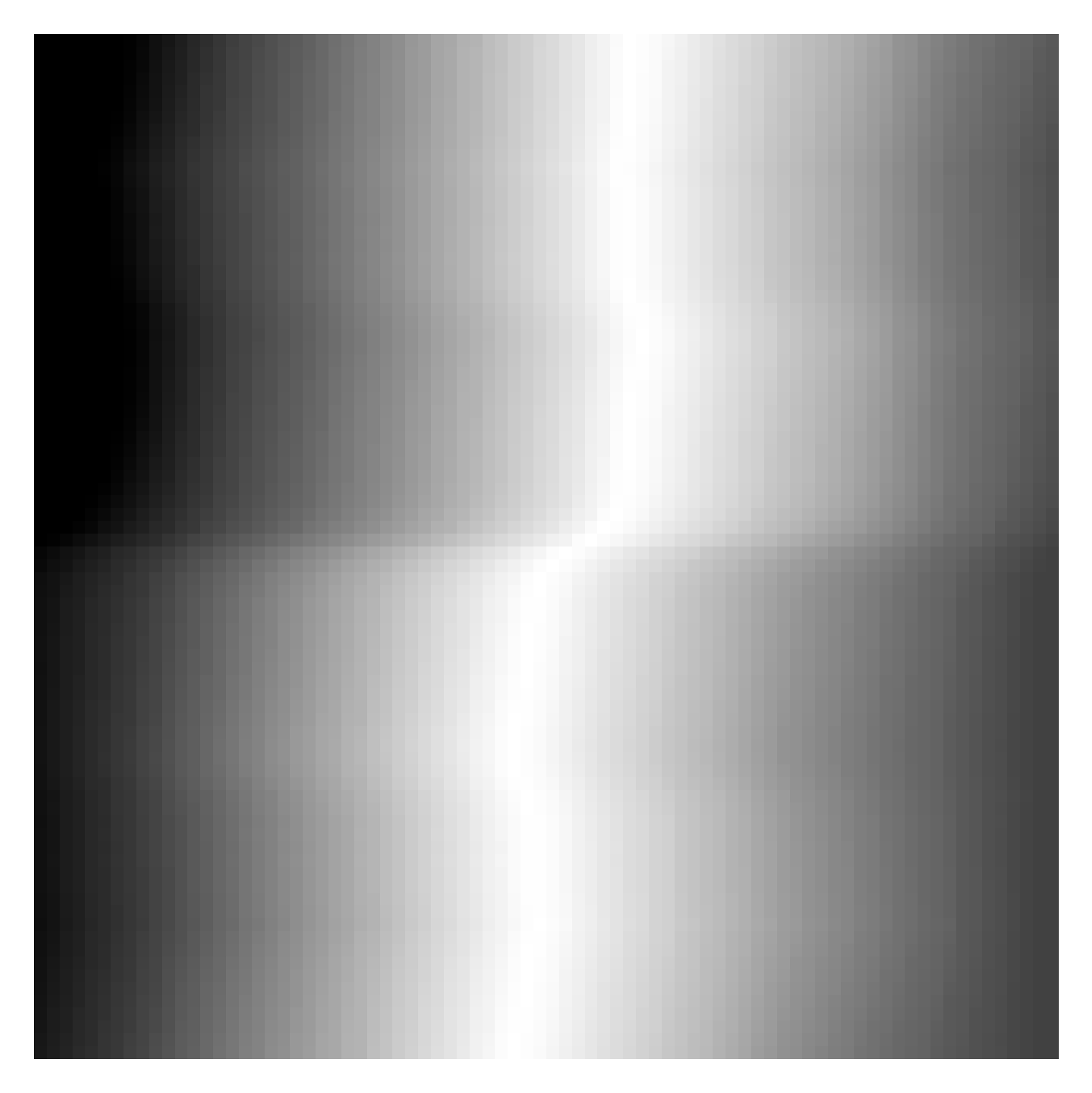}
\label{fig: rectangle deformed zeroth-order}}
\subfigure[]{
\includegraphics[height = 2.6cm]{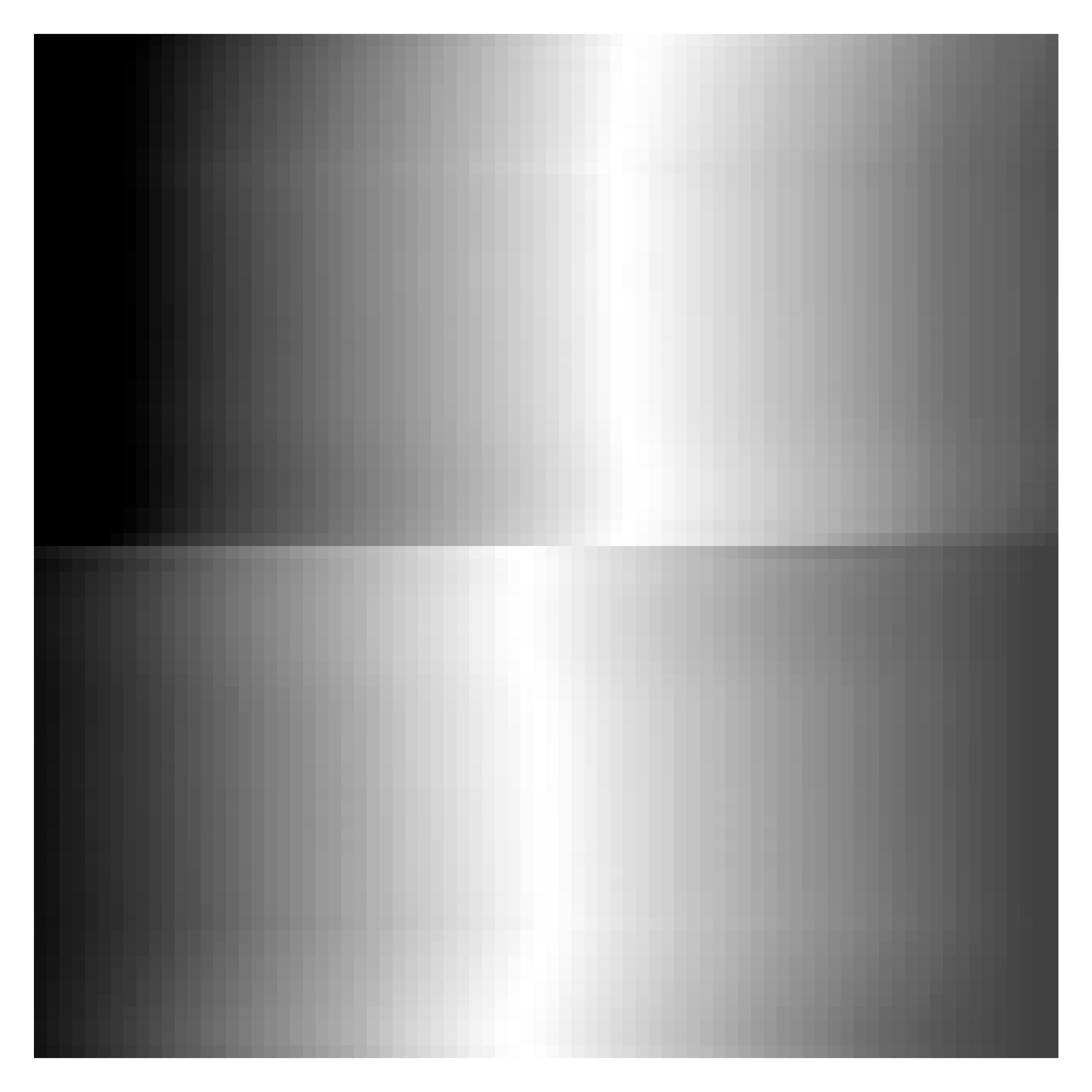}
\label{fig: rectangle deformed zeroth and first order}}

\subfigure[]{
\includegraphics[height = 2.6cm]{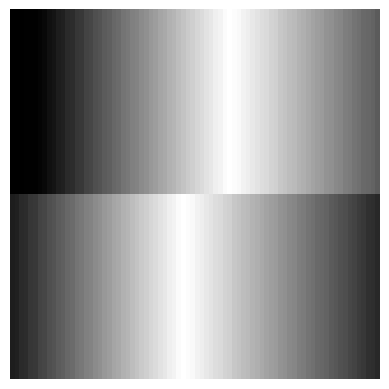}
\label{fig: rectangle reference}}
\subfigure[]{
\includegraphics[height = 2.6cm]{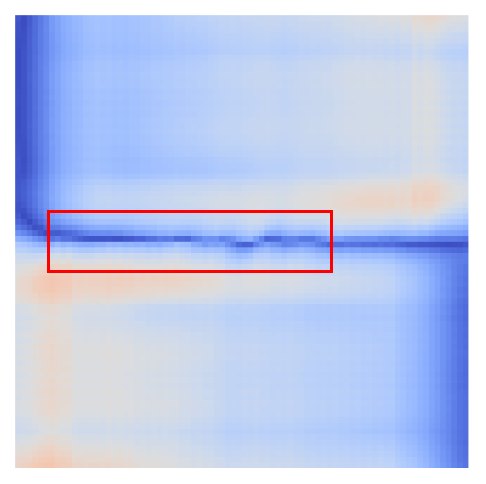}
\label{fig: rectangle deformed mag gaussian}}
\subfigure[]{
\includegraphics[height = 2.6cm]{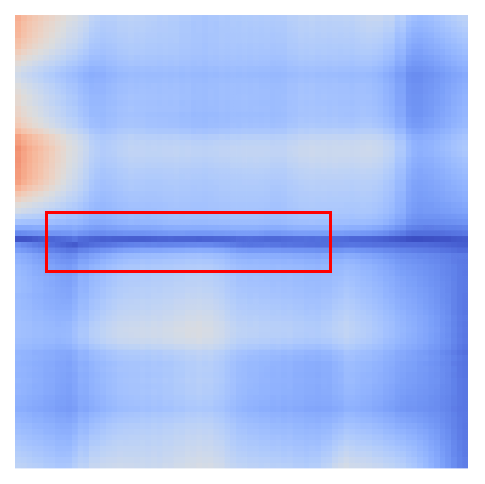}
\label{fig: rectangle deformed mag zeroth-order}}
\subfigure[]{
\includegraphics[height = 2.6cm]{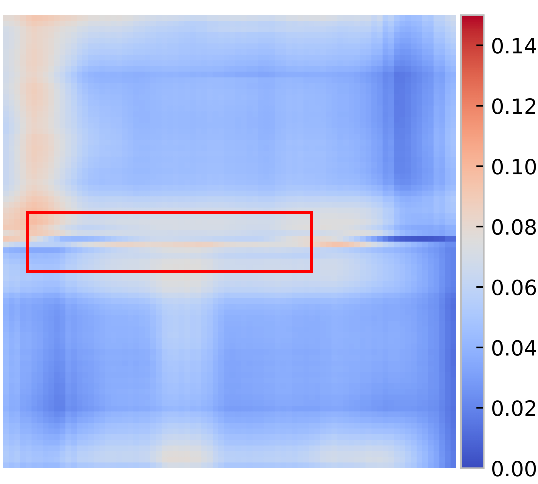}
\label{fig: rectangle deformed mag zeroth and first order}}
\caption{\label{fig: rectangle result}Registration results on synthetic rectangle images with sliding motion. \subref{fig: rectangle template} Template image; \subref{fig: rectangle deformed gaussian} Deformed template image using Gaussian kernel; \subref{fig: rectangle deformed zeroth-order} Deformed template image using zeroth-order momentum based on Wendland kernel; \subref{fig: rectangle deformed zeroth and first order} Deformed template image using both zeroth- and first-order momenta based on Wendland kernel; \subref{fig: rectangle reference} Reference image; \subref{fig: rectangle deformed mag gaussian} Deformation field magnitudes obtained by using Gaussian kernel; \subref{fig: rectangle deformed mag zeroth-order} Deformation field magnitudes obtained by using zeroth-order momentum based on Wendland kernel; \subref{fig: rectangle deformed mag zeroth and first order} Deformation field magnitudes obtained by using both zeroth- and first-order momenta based on Wendland kernel.}
\end{figure} 

The qualitative results are shown in \cref{fig: rectangle result}. The standard LDDMM registration with Gaussian kernel causes more smoothing deformation around the sliding interfaces (\cref{fig: rectangle deformed gaussian}), and the corresponding deformation field magnitude (\cref{fig: rectangle deformed mag gaussian}) also exhibits smoother deformation, as indicated by the red box. Using only the zeroth-order momentum method results in a blurred sliding interface (\cref{fig: rectangle deformed zeroth-order,fig: rectangle deformed mag zeroth-order}). However, when both zeroth- and first-order momenta are used, the boundary is better preserved, and the deformation field is much more realistic (\cref{fig: rectangle deformed zeroth and first order,fig: rectangle deformed mag zeroth and first order}).

\begin{figure}[htbp]
\centering
\subfigure[]{
\includegraphics[height = 2.6cm]{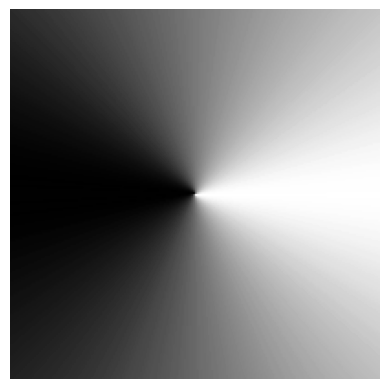}
\label{fig: wheel template}}
\subfigure[]{
\includegraphics[height = 2.6cm]{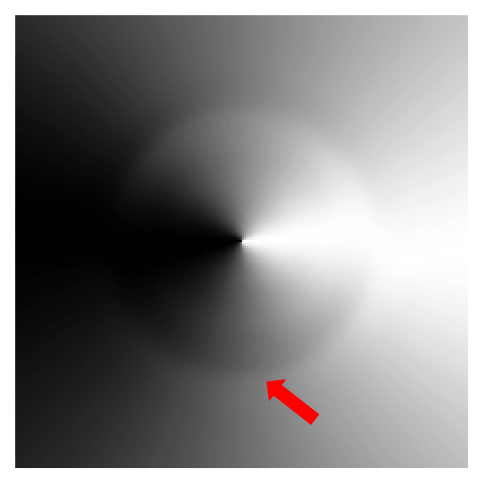}
\label{fig: wheel deformed gaussian}}
\subfigure[]{
\includegraphics[height = 2.6cm]{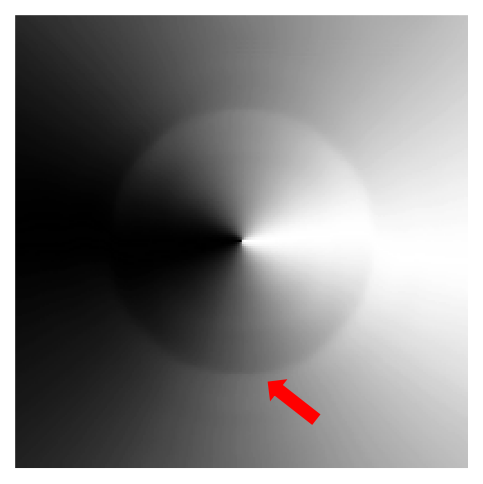}
\label{fig: wheel deformed zeroth-order}}
\subfigure[]{
\includegraphics[height = 2.6cm]{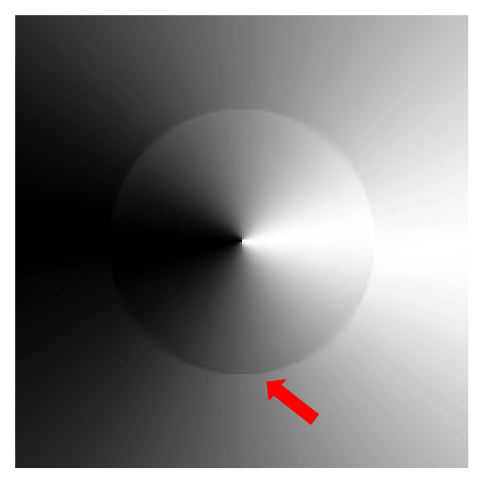}
\label{fig: wheel deformed zeroth and first order}}

\subfigure[]{
\includegraphics[height = 2.6cm]{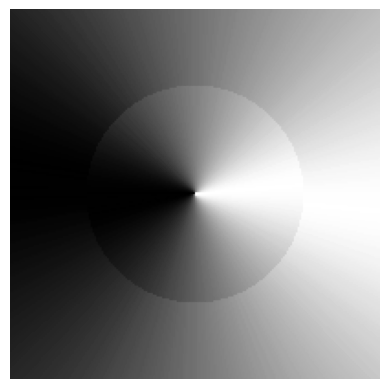}
\label{fig: wheel reference}}
\subfigure[]{
\includegraphics[height = 2.6cm]{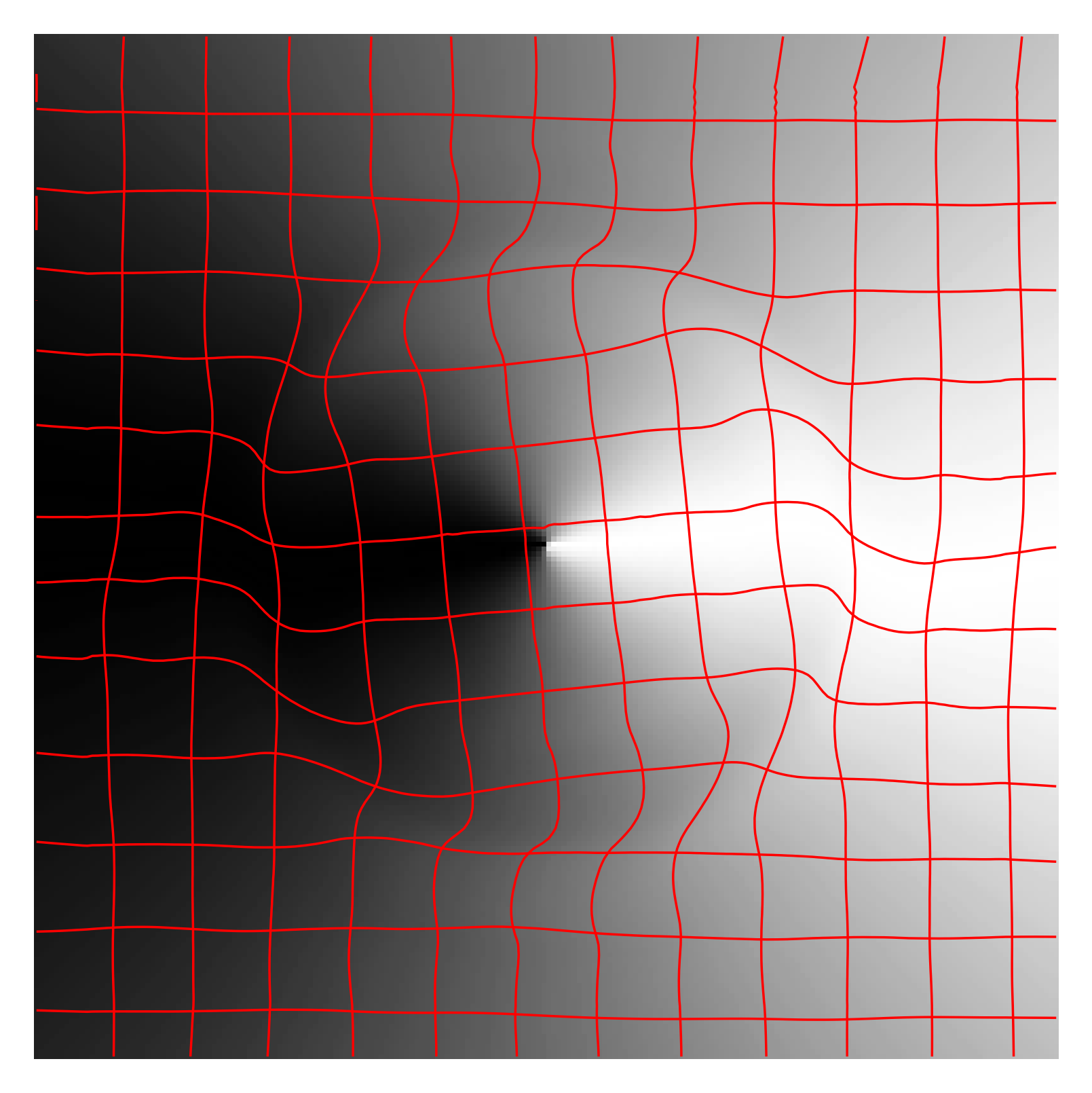}
\label{fig: wheel deformed grid gaussian}}
\subfigure[]{
\includegraphics[height = 2.6cm]{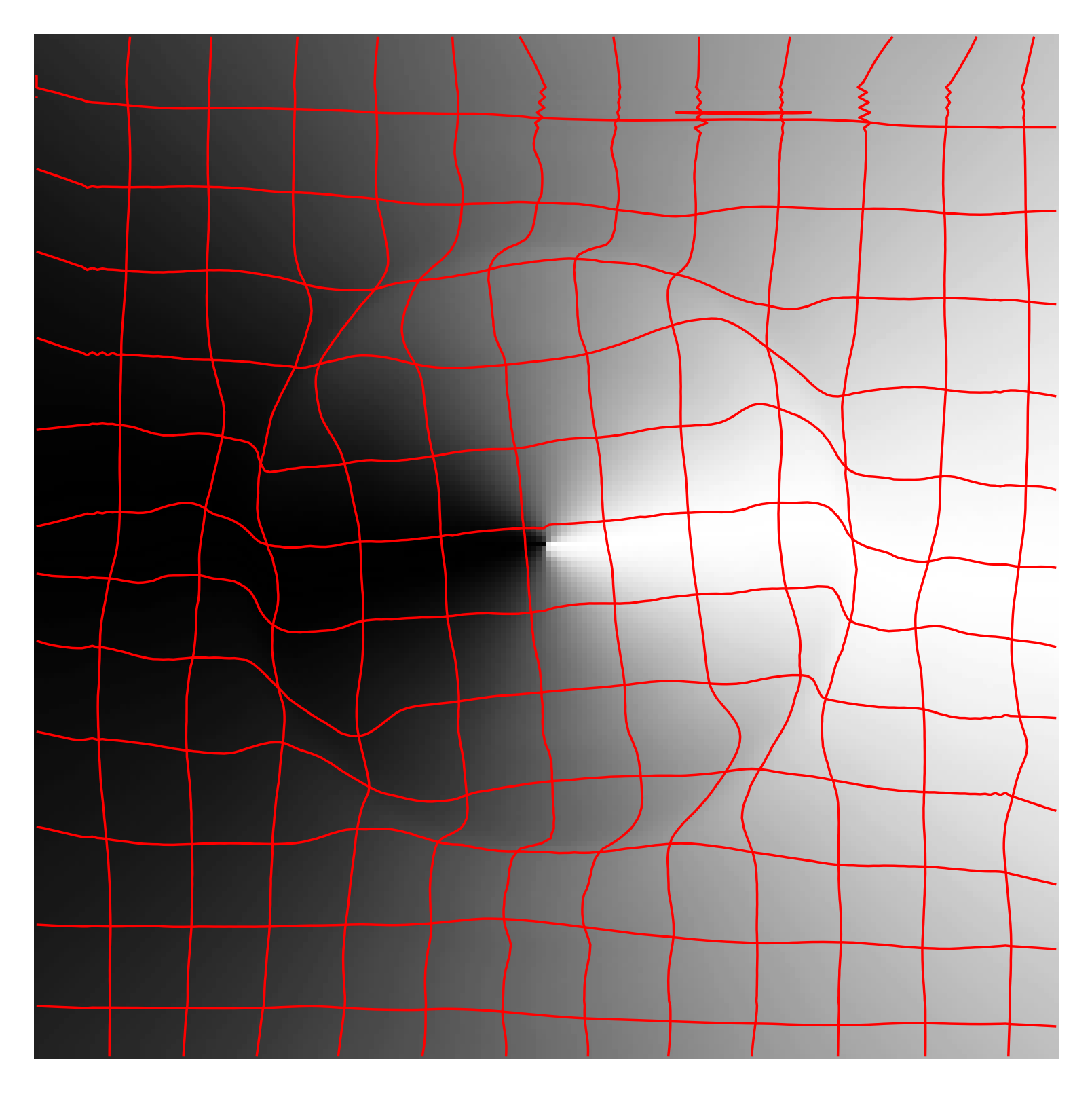}
\label{fig: wheel deformed grid zeroth-order}}
\subfigure[]{
\includegraphics[height = 2.6cm]{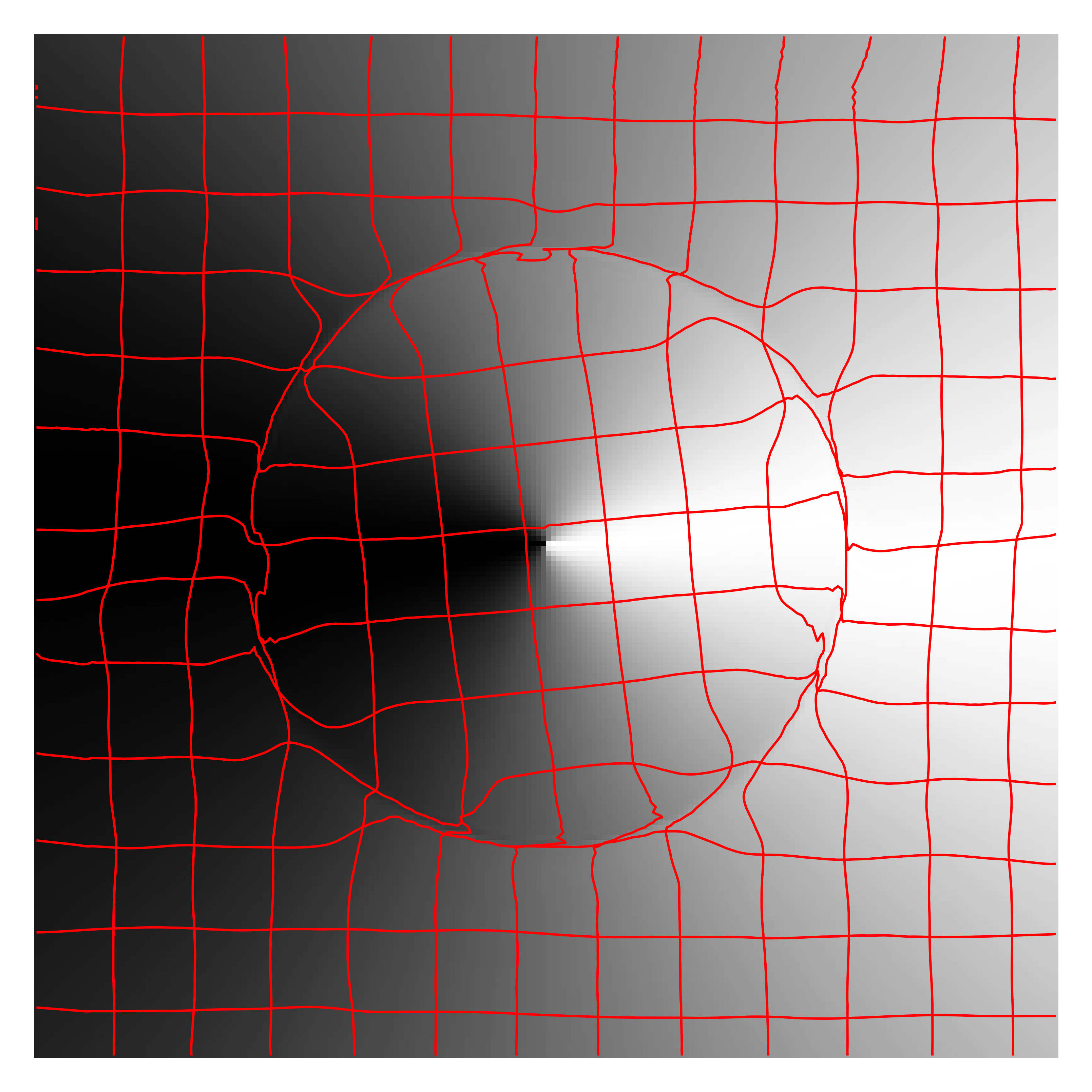}
\label{fig: wheel deformed grid zeroth and first order}}
\caption{\label{fig: wheel result}Registration results on synthetic wheel images with sliding motion. \subref{fig: wheel template} Template image; \subref{fig: wheel deformed gaussian} Deformed template image using Gaussian kernel; \subref{fig: wheel deformed zeroth-order} Deformed template image using zeroth-order momentum based on Wendland kernel; \subref{fig: wheel deformed zeroth and first order} Deformed template image using both zeroth- and first-order momenta based on Wendland kernel; \subref{fig: wheel reference} Reference image; \subref{fig: wheel deformed grid gaussian} Deformed grid obtained by using Gaussian kernel; \subref{fig: wheel deformed grid zeroth-order} Deformed grid obtained by using zeroth-order momentum based on Wendland kernel; \subref{fig: wheel deformed grid zeroth and first order} Deformed grid obtained by using both zeroth- and first-order momenta based on Wendland kernel.}
\end{figure}

In the second experiment, we tested the method on a more complex sliding motion. The template image was deformed by rotating the inner and outer circular regions by 5 degrees in opposite directions, as shown in \cref{fig: wheel template,fig: wheel reference}. When using the standard LDDMM with Gaussian kernel, the inner and outer parts are aligned smoothly (\cref{fig: wheel deformed gaussian}), and the corresponding deformed grid (\cref{fig: wheel deformed grid gaussian}) also exhibit this. 
while the multiplicative Wendland kernel based method with only zeroth-order momentum improved registration accuracy (\cref{fig: wheel deformed zeroth-order,fig: wheel deformed grid zeroth and first order}), it still failed to preserve the boundary and deal with sliding motion.
Benefiting from the derivative kernel and first-order momentum, our proposed method could capture the sliding motion near the boundary and provide the best result, as indicated by the red arrow (\cref{fig: wheel deformed zeroth and first order,fig: wheel deformed grid zeroth and first order}).

\subsection{DIR-Lab dataset}
\label{subsec: DIR-Lab dataset}
\begin{table}[!htbp]
	\centering
	\caption{Comparison of TRE (target registration error) on the DIRLAB dataset. Bold values are the best values (the lowest error).}
	\label{eq: TRE}
	\begin{tabular}{lccccc}
		\toprule
		& & \multicolumn{3}{c}{After registration}  \\
		\cmidrule(r){3-5} 
		Case &Before registration  & Gaussian   & Wendland  & proposed \\
		\midrule
		4DCT01   &3.89     &0.84    &0.83    &\textbf{0.77}  \\
		4DCT02   &4.34     &0.81    &\textbf{0.78}    &0.79  \\
		4DCT03   &6.94     &\textbf{1.32}   &1.35   &\textbf{1.32} \\
		4DCT04   &9.83     &\textbf{1.64}    &\textbf{1.64}    &1.77  \\
		4DCT05   &7.48     &2.17    &2.20   &\textbf{2.15}  \\
		4DCT06   &10.89    &2.39    &\textbf{2.36}    &2.66  \\
		4DCT07   &11.03    &3.43    &\textbf{3.40}    &3.42  \\
		4DCT08   &14.99    &\textbf{8.54}    &8.55    &8.86  \\
		4DCT09   &7.92     &2.10    &2.09    &\textbf{1.47}  \\
		4DCT10   &7.30     &2.14    &2.15    &\textbf{1.71}  \\
		\hline
		Mean   &8.46     &2.54  &2.54  &\textbf{2.49} \\
		\bottomrule
	\end{tabular}
\end{table}
In this section, we evaluate the performance of our proposed algorithm on the publicly available 4D Lung CT DIRLAB dataset (\url{http://www.dir-lab.com/}). The dataset contains 10 4D-CT cases, each consisting of 10 3D-CT breathing sequences with manually 300 annotated landmarks.
The spatial resolution ranges from $0.97 \times 0.97 \times 2.5 ~mm^3$ to $1.16 \times 1.16 \times 2.5 ~mm^3$. In our experiments, we cropped the images to contain the thoracic cavity and clipped the image intensities between 50 and 1200 HU. We used the extreme inhale images as the reference images and the extreme exhale images as the template images. 

\begin{figure}[!htbp]
\centering
\subfigure[]{
\includegraphics[height = 2.6cm]{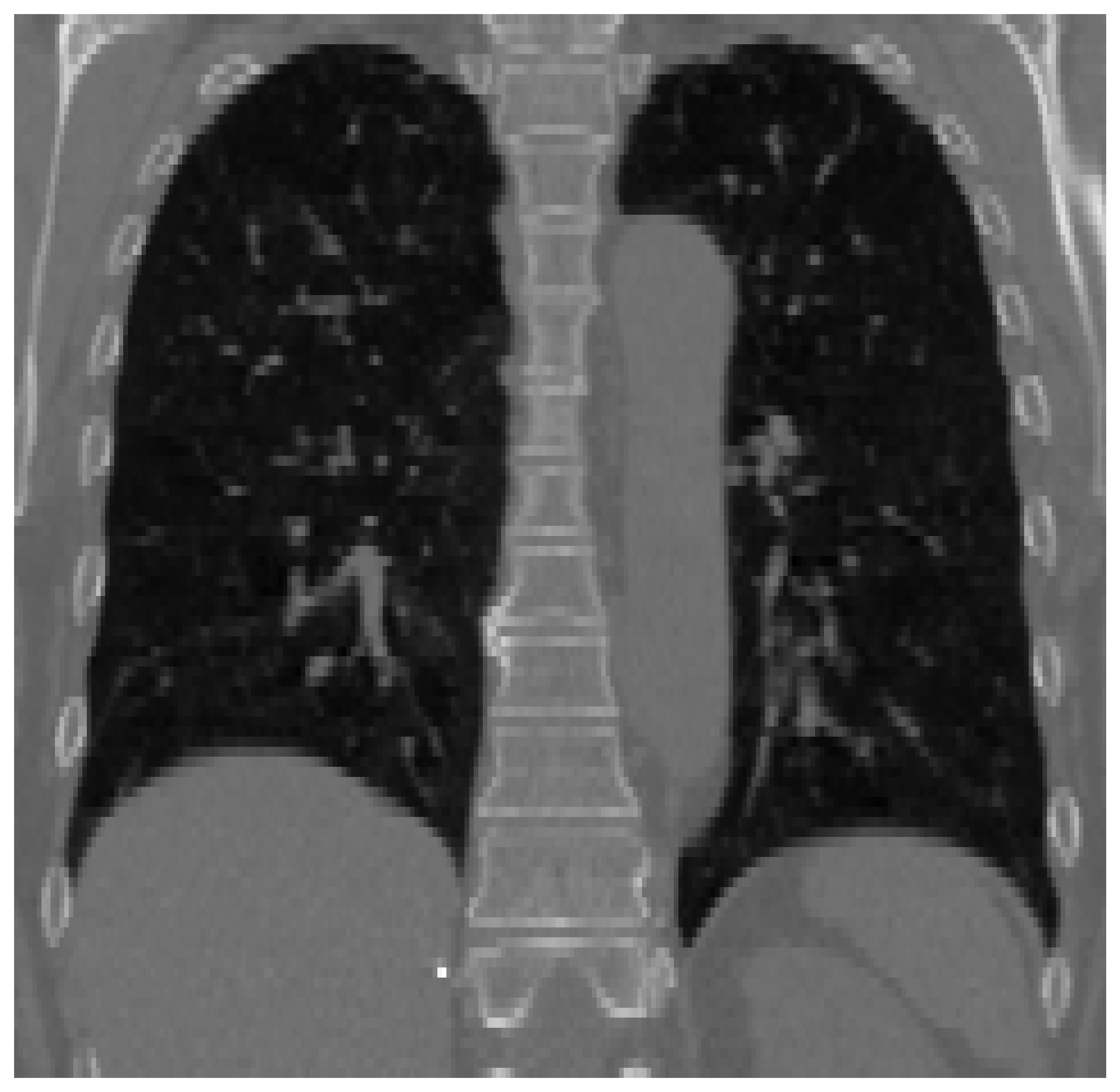}
\label{fig: lung template}}
\subfigure[]{
\includegraphics[height = 2.6cm]{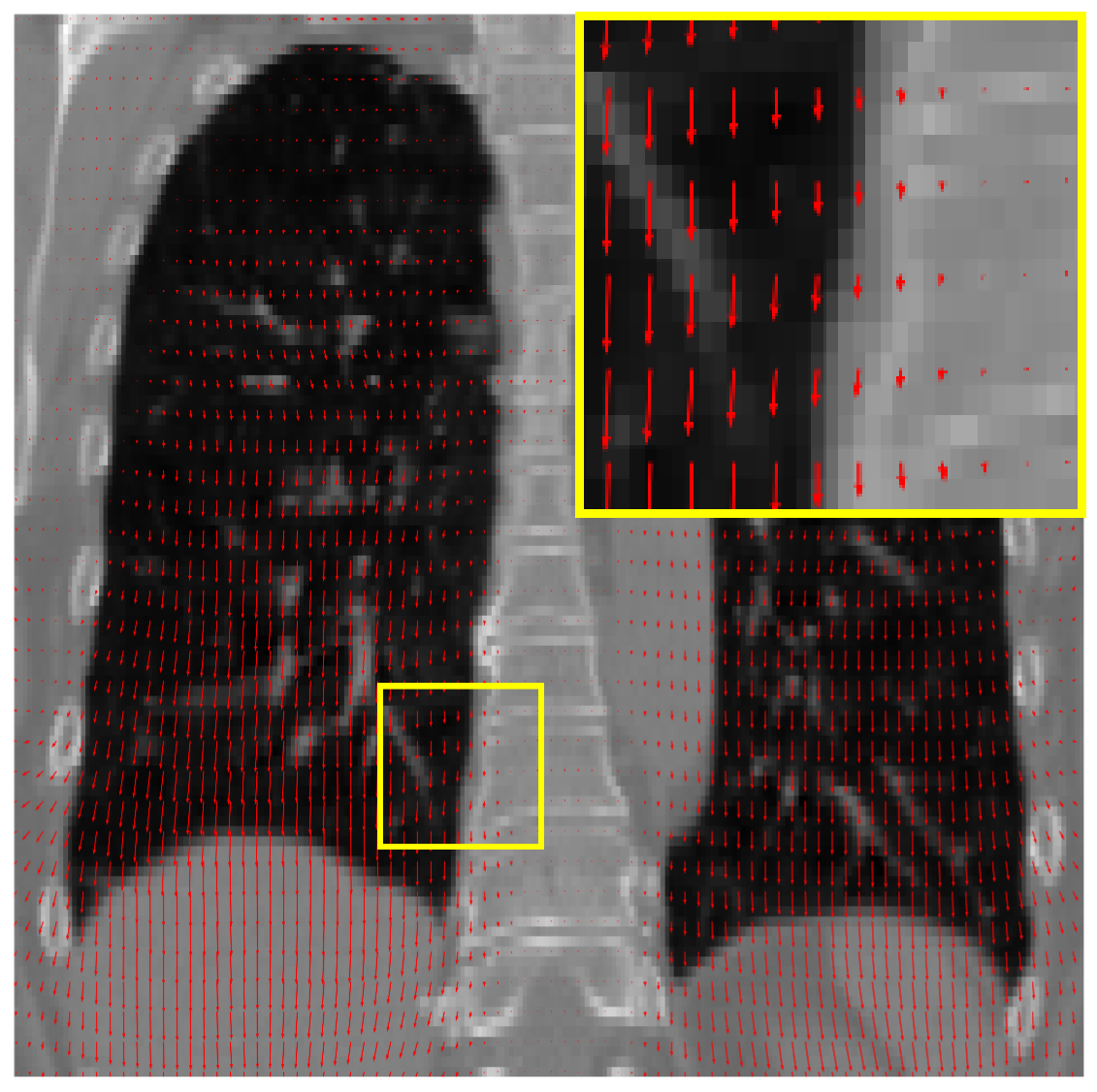}
\label{fig: lung deformed gaussian}}
\subfigure[]{
\includegraphics[height = 2.6cm]{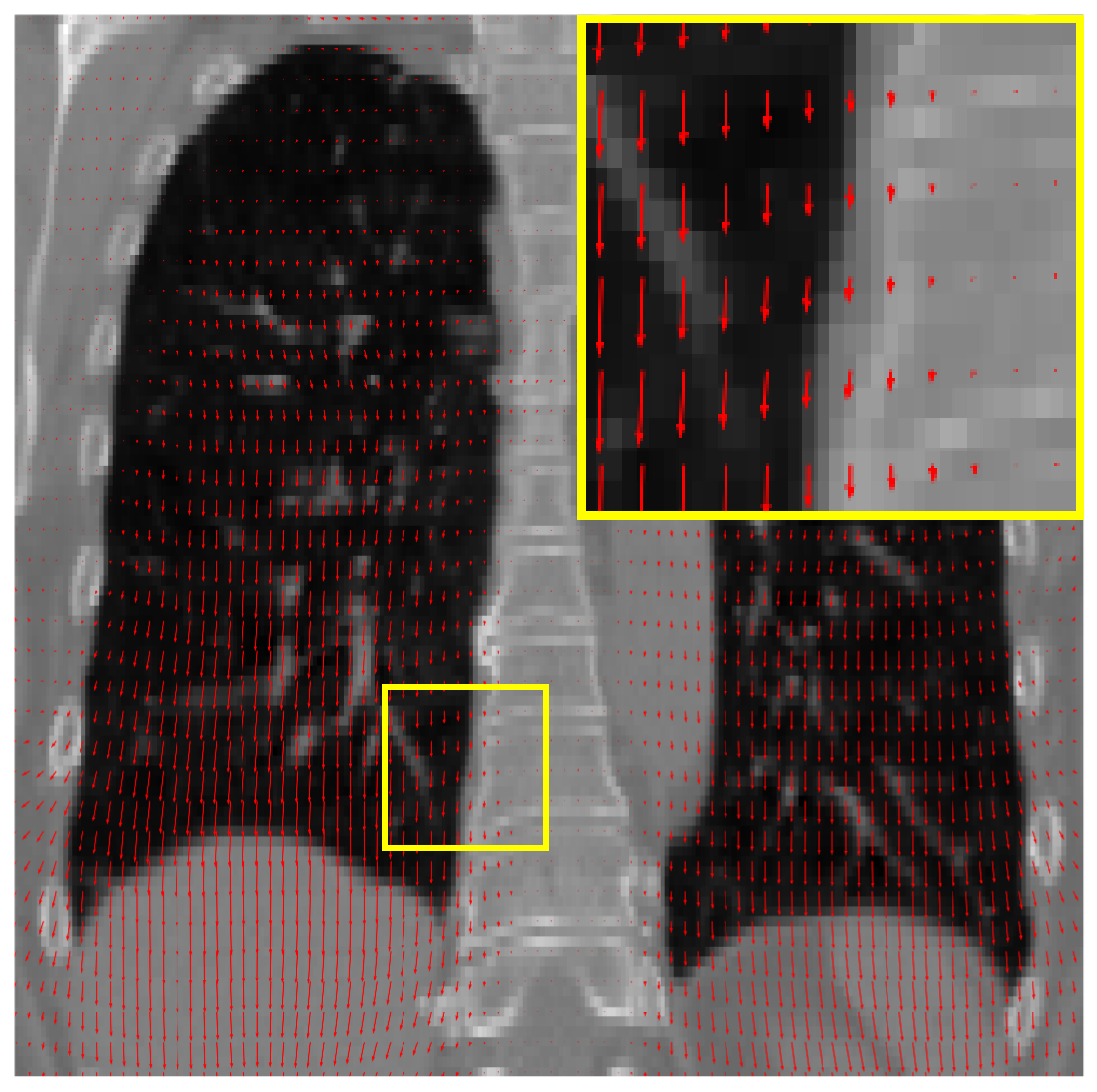}
\label{fig: lung deformed zeroth-order}}
\subfigure[]{
\includegraphics[height = 2.6cm]{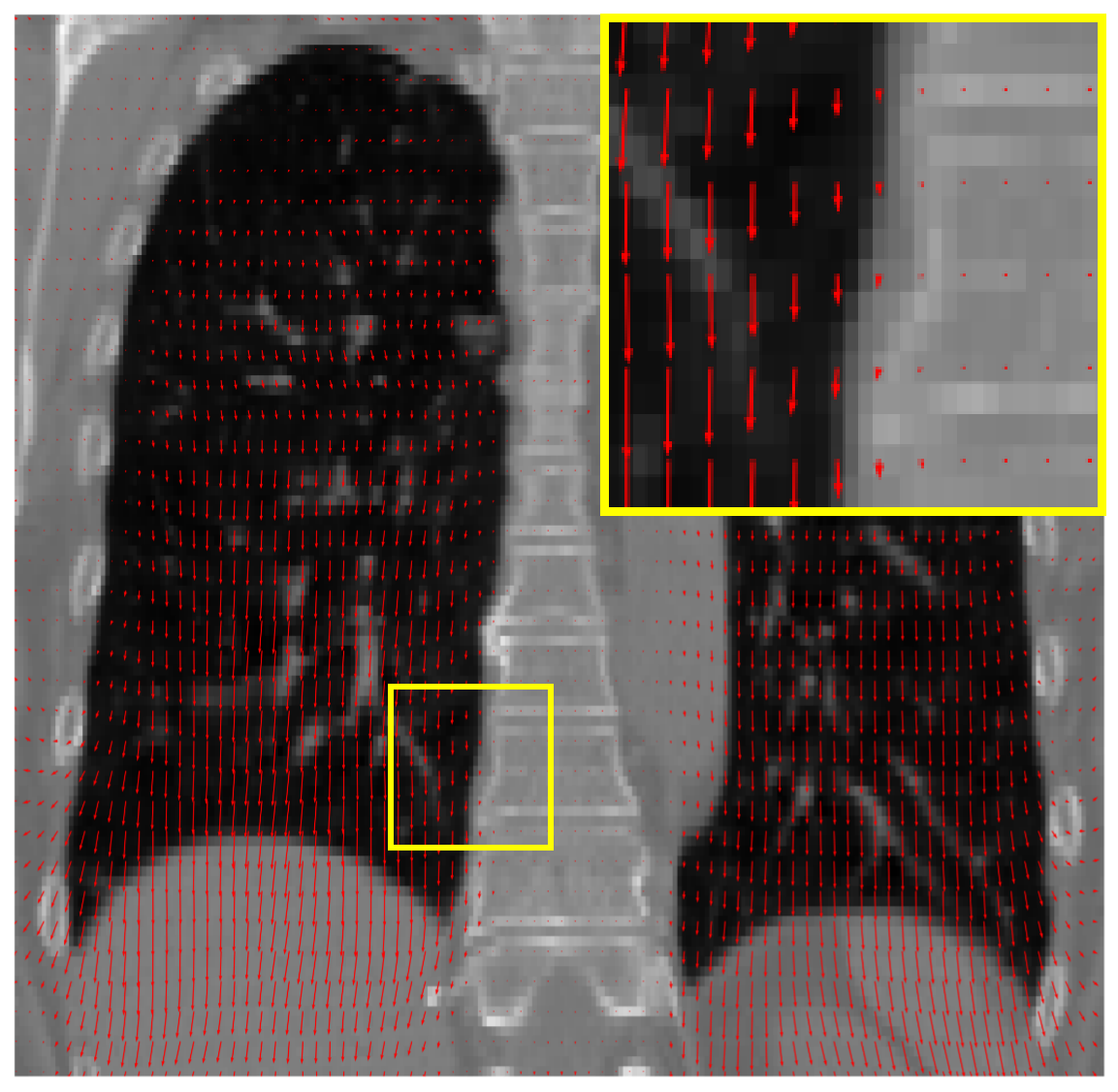}
\label{fig: lung deformed zeroth and first order}}

\subfigure[]{
\includegraphics[height = 2.6cm]{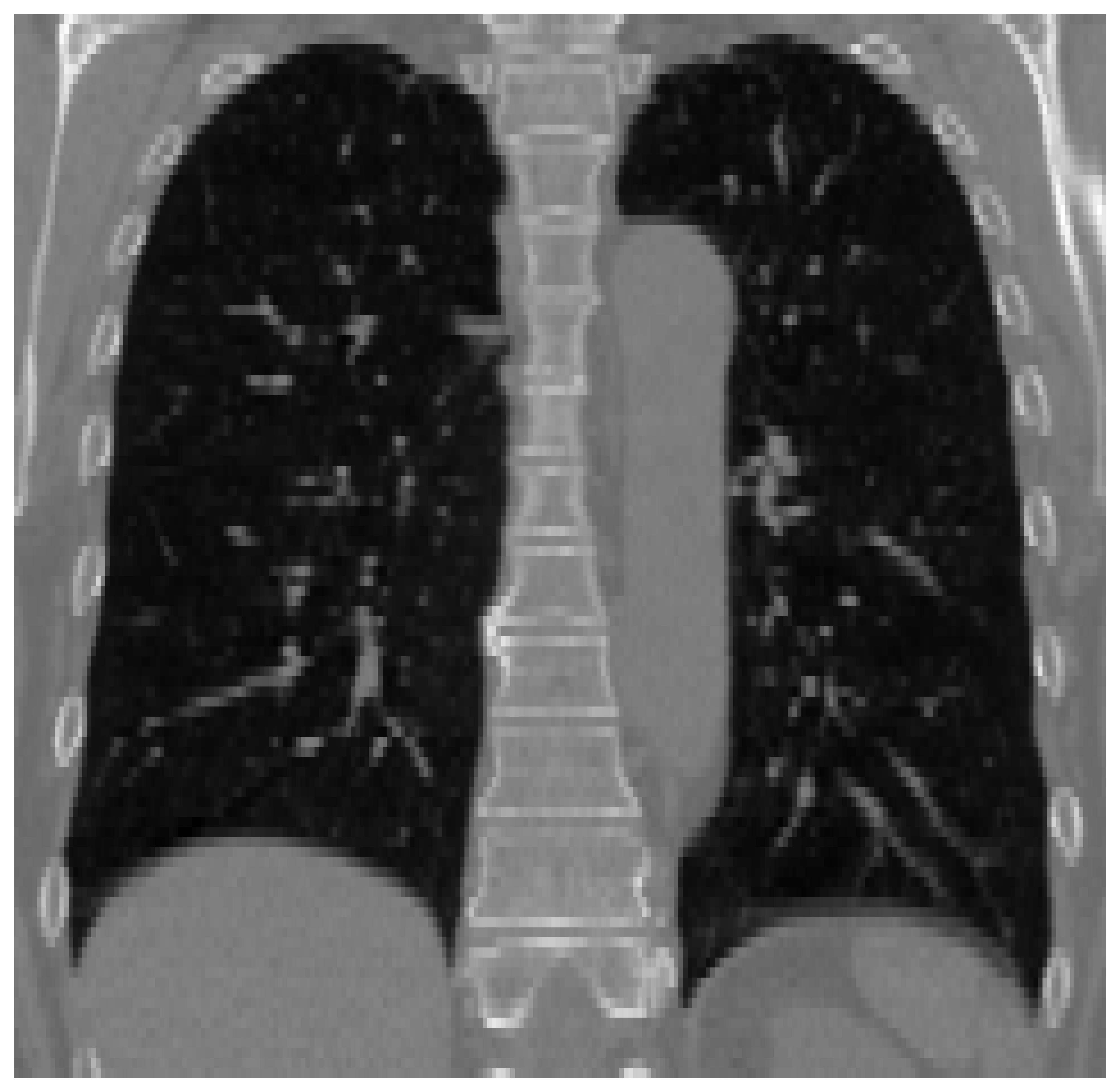}
\label{fig: lung reference}}
\subfigure[]{
\includegraphics[height = 2.6cm]{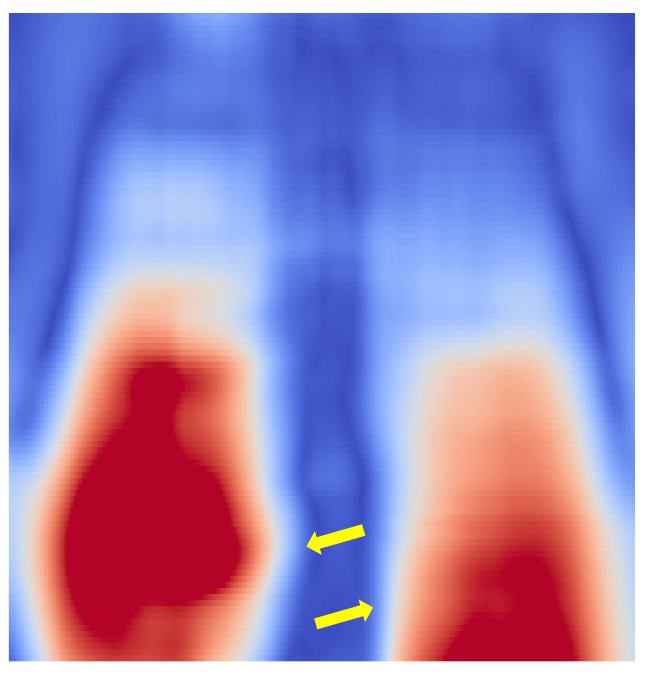}
\label{fig: lung deformed mag gaussian}}
\subfigure[]{
\includegraphics[height = 2.6cm]{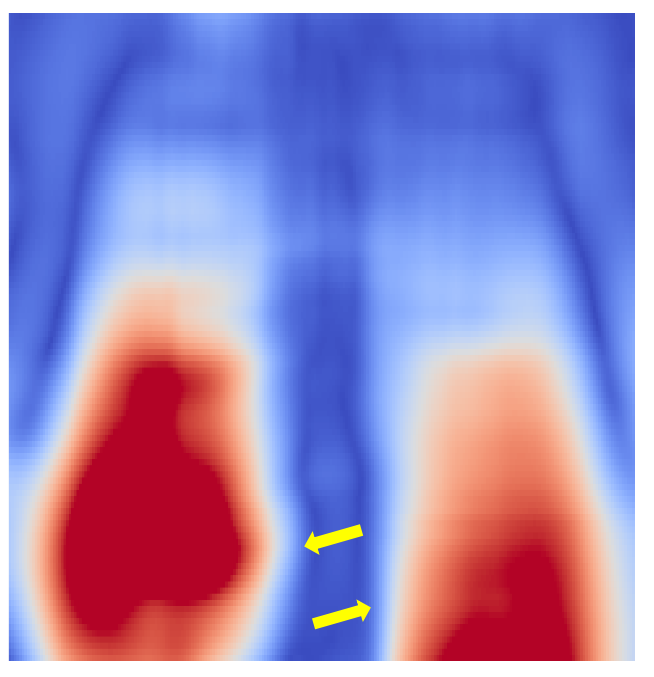}
\label{fig: lung deformed mag zeroth-order}}
\subfigure[]{
\includegraphics[height = 2.6cm]{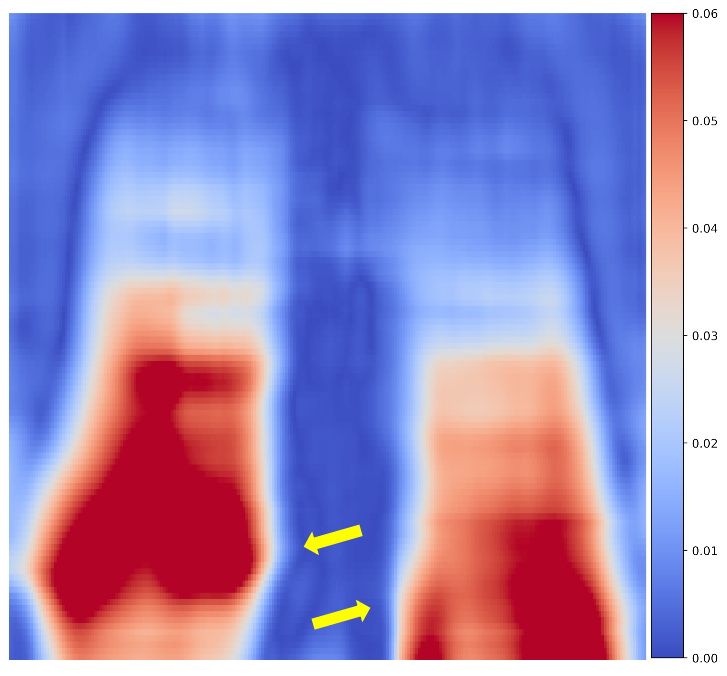}
\label{fig: lung deformed mag zeroth and first order}}
\caption{\label{fig: lung result}Registration results on lung images with sliding motion. \subref{fig: lung template} Template image; \subref{fig: lung deformed gaussian}-\subref{fig: lung deformed zeroth and first order} Deformed template image overlaid with red deformation field by using Gaussian kernel, Wendland kernel with zeroth-order momentum and both zeroth- and first-order momenta, respectively; \subref{fig: lung reference} Reference image; \subref{fig: lung deformed mag gaussian}-\subref{fig: lung deformed mag zeroth and first order} Deformation field magnitudes obtained by using Gaussian kernel, Wendland kernel with zeroth-order momentum and both zeroth- and first-order momenta, respectively.}
\end{figure}

To quantify the accuracy of the registration, the Target Registration Error (TRE) is calculated for 300 anatomical landmarks between reference and template landmarks, both before and after registration. The results, shown in \cref{eq: TRE}, indicated that our proposed method, as well as the LDDMM using Gaussian and Wenland kernels, produced statistically significant improvements in terms of TRE compared to before registration. Furthermore, our method performed similarly to the tested methods in terms of TRE.


\begin{figure}[!htbp]
\centering
\subfigure[]{
\includegraphics[height = 4cm]{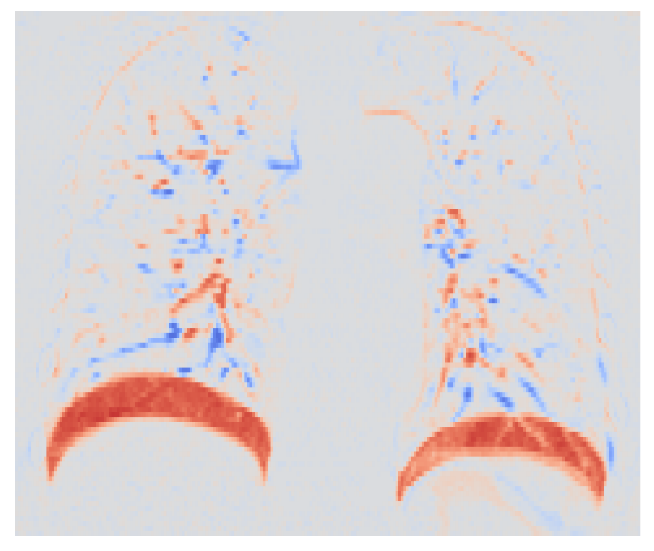}
\label{fig: lung diff before}}
\subfigure[]{
\includegraphics[height = 4cm]{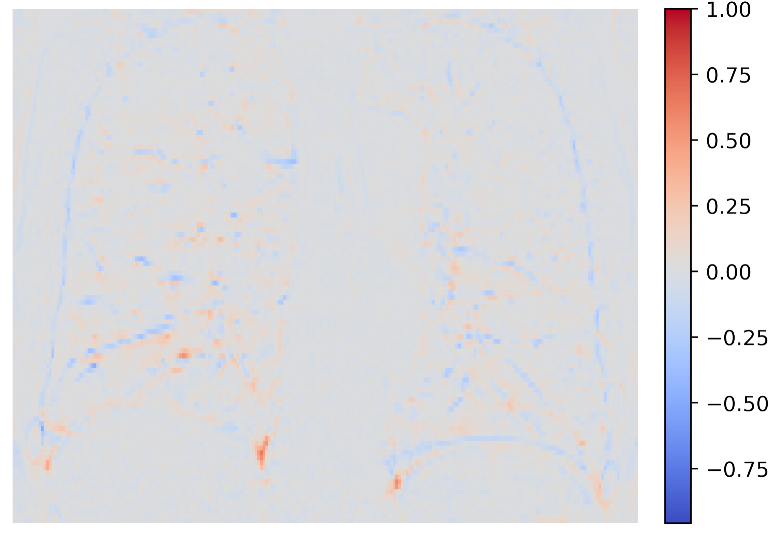}
\label{fig: lung diff after}}
\caption{\label{fig: lung diff}Registration difference images before and after registration. \subref{fig: lung diff before} Difference image before registration; \subref{fig: lung diff after} Difference image after registration using proposed method.}
\end{figure}

A visual inspection of the deformation fields, as shown in \cref{fig: lung result} for case 5, revealed that our proposed method produced a more physiologically plausible sliding motion between the thoracic cavity and the lung, particularly in the yellow region of interest (highlighted by the yellow box in \cref{fig: lung deformed zeroth and first order}). In contrast, as depicted in \cref{fig: lung deformed gaussian,fig: lung deformed zeroth-order}, the vertebra is torn down in an unnatural manner along with the lung. The yellow arrows in the deformation field magnitude images in \cref{fig: lung deformed mag gaussian,fig: lung deformed mag zeroth-order,fig: lung deformed mag zeroth and first order} indicated that our proposed method generated more discontinuous deformation fields. 
The difference image before and after registration using our proposed method, shown in \cref{fig: lung diff}, demonstrated that our method can achieve an accurate registration result.

\section{Conclusion}
\label{sec:Conclusion}
In this paper, we proposed a novel registration method that combines both zeroth- and first-order momenta based on a multiplicative Wendland kernel within the LDDMM framework. Our method was supported by a mathematical analysis of the derived flow, providing a deeper understanding of the method. The experimental results on synthetic 2D images and a clinical CT dataset showed that our method can effectively preserve discontinuous motions at sliding interfaces. These results demonstrate the potential of our proposed method for improved registration accuracy in medical imaging applications.


\bibliographystyle{siamplain}
\bibliography{references}

\end{document}